\documentclass[11pt, a4paper]{amsart}
\usepackage{amssymb, amsthm}

\usepackage{enumerate}

\newenvironment{nouppercase}{%
\renewcommand{\uppercasenonmath}[1]{}}{}

\makeatletter
\renewcommand{\section}{%
\@startsection{section}{1}%
  \z@{.7\linespacing\@plus\linespacing}{.5\linespacing}%
 {\normalfont\large\bfseries\centering}}
\makeatother

\newtheorem{theorem}{Theorem}[section]

\newtheorem{lemma}[theorem]{Lemma}
\newtheorem{proposition}[theorem]{Proposition}

\theoremstyle{definition}
\newtheorem{definition}[theorem]{Definition}

\theoremstyle{remark}
\newtheorem{remark}[theorem]{Remark}
\newtheorem*{remark*}{Remark}

\newcommand{\R}{\mathbb{R}}
\def\C{\mathbb{C}}
\def\x{\times}
\newcommand{\pt}{\partial}
\newcommand{\del}{\partial}
\newcommand{\om}{\omega}
\newcommand{\lam}{\lambda}
\newcommand{\ol}{\overline}
\newcommand{\on}{\operatorname}
\renewcommand{\Re}{\operatorname{Re}}
\renewcommand{\Im}{\operatorname{Im}}
\newcommand{\spn}{\operatorname{span}}
\newcommand{\ce}{\mathrel{\mathop:}=}
\newcommand{\ec}{=\mathrel{\mathop:}} 
\newcommand{\ve}{\varepsilon}
\newcommand{\eps}{\varepsilon}
\newcommand{\mb}{\mathbb}
\newcommand{\mc}{\mathcal}
\newcommand{\mr}{\mathrm}
\newcommand{\bal}{\bar{\alpha}}

\def\tbra[#1,#2]{\langle #1 , #2\rangle}
\def\rbra[#1,#2]{\left( #1 , #2 \right)}
\newcommand{\cleq}{\lesssim}
\newcommand{\cgeq}{\gtrsim}
\let\del\partial

\renewcommand{\l}{\left}
\renewcommand{\r}{\right}

\newcommand{\cE}{{\mathcal E}}
\newcommand{\cM}{{\mathcal M}}
\newcommand{\ds}{\displaystyle}

\begin{document}

\title[]{\Large 
Instability of degenerate solitons for nonlinear Schr\"odinger equations with derivative}

\author[]{\large
Noriyoshi Fukaya}
\address[N. Fukaya]{\upshape 
Department of Mathematics, Tokyo University of Science, Tokyo, 162-8601, Japan}
\email{fukaya@rs.tus.ac.jp}

\author[]{\large
Masayuki Hayashi}
\address[M. Hayashi]{\upshape 
Research Institute for Mathematical Sciences, Kyoto University, Kyoto 606-8502, Japan}

\email{hayashi@kurims.kyoto-u.ac.jp}



\date{\today}

\begin{abstract}
We consider the following nonlinear Schr\"{o}dinger equation with derivative:
\begin{equation}
\label{eq:1}
i u_t =-u_{xx} -i |u|^{2}u_x -b|u|^4u , \quad (t,x) \in \R \times \R, \ b \in \R .
\end{equation}
If $b=0$, this equation is a gauge equivalent form of the well-known derivative nonlinear Schr\"{o}dinger (DNLS) equation. 
The soliton profile of DNLS satisfies a certain double power elliptic equation with cubic-quintic nonlinearities. The quintic nonlinearity in \eqref{eq:1} only affects the coefficient in front of the quintic term in the elliptic equation, so in this sense the additional nonlinearity is natural as a perturbation preserving soliton profiles of DNLS.
When $b\ge 0$, the equation \eqref{eq:1} has degenerate solitons whose momentum and energy are zero, and if $b=0$, they are algebraic solitons. 
Inspired from the works \cite{MM01, FHR19} on instability theory of the $L^2$-critical generalized KdV equation, we study the instability of degenerate solitons of \eqref{eq:1} in a qualitative way, and when $b>0$, we obtain a large set of initial data yielding the instability. The arguments except one step in our proof work for the case $b=0$ in exactly the same way, which is a small step towards understanding the dynamics around algebraic solitons of the DNLS equation.
\end{abstract}

\begin{nouppercase}
\maketitle
\end{nouppercase}


\numberwithin{equation}{section} 
\section{Introduction}


We consider the following nonlinear Schr\"odinger equation with derivative:
\begin{align} \label{DNLSb}
  iu_t
  =-u_{xx}
  -i|u|^2u_x
  -b|u|^4u,
  \quad(t,x)\in\R\x\R,
\end{align}
where $b\in\R$, and $u$ is the complex-valued unknown function of $(t,x)\in\R\x\R$. 
It is well-known (see \cite{Oz96}) that \eqref{DNLSb} is locally well-posed in the energy space $H^1(\R)$ 
and the following three quantities
\begin{align*}
\tag{Energy}
  E(u)
  &\ce\frac{1}{2}\|u_x\|_{L^2}^2
  -\frac{1}{4}(i|u|^2u_x, u)_{L^2}
  -\frac{b}{6}\|u\|_{L^6}^6, 
\\ \tag{Mass}
  M(u)
  &\ce\|u\|_{L^2}^2, 
\\ \tag{Momentum}
  P(u)
  &\ce(iu_x,u)_{L^2},
\end{align*}
are conserved by the flow. Here the inner product $(\cdot,\cdot)_{L^2}$ is defined by
\[(v,w)_{L^2}
  =\Re\int_{\R}v(x)\ol{w(x)}\,dx, \]
and we regard $L^2(\R)$ as a real Hilbert space. The equation \eqref{DNLSb} is $L^2$-critical (mass-critical) in the sense that \eqref{DNLSb} is invariant under the scaling
\[u_\lambda(t,x)
  =\lambda^{1/2}u(\lambda^2t,\lambda x), \]
which satisfies $\|u_\lambda(0)\|_{L^2}=\|u(0)\|_{L^2}$. By using the energy functional, \eqref{DNLSb} is rewritten as
\begin{equation*}
  iu_t(t)
  =E'(u(t)). 
\end{equation*}

When $b=0$ the equation \eqref{DNLSb}
is sometimes referred to as the Chen-Lee-Liu equation \cite{CLL79}. This equation is a gauge equivalent form of the well-known derivative nonlinear Schr\"odinger equation
\begin{align} 
\label{DNLS}
\tag{DNLS}
  i\psi_t
  =-\psi_{xx}
  -i(|\psi|^2\psi)_x,\quad
  (t,x)\in\R\x\R,
\end{align}
which was introduced as a model in plasma physics \cite{MOMT76, M76} and shown to be completely integrable \cite{KN78}.
The soliton profile of \eqref{DNLS} satisfies a double power elliptic equation with cubic-quintic nonlinearities (see \eqref{eq:1.4}). The quintic nonlinearity in \eqref{DNLSb} only affects the coefficient in front of the quintic term in the elliptic equation, so in this sense the additional nonlinearity is not artificial, or rather natural as a perturbation preserving soliton profiles of \eqref{DNLS}. 
We note that the equation \eqref{DNLSb} is not integrable in the case $b\neq0$ while preserving the $L^2$-critical structure of \eqref{DNLS}.
This means that the equation \eqref{DNLSb} can be seen as an important model to clarify the difference between integrable and nonintegrable cases in the $L^2$-critical framework. 

Regardless of the relevance to \eqref{DNLS}, the equation \eqref{DNLSb} itself is an interesting mathematical model possessing a two-parameter family of solitons.\footnote{The terminology \textit{soliton} was originally used in a context of integrable equations, but we also use it for nonintegrable equations according to conventions in the literature.} 
For example, when $b>0$, this equation possesses both stable and unstable solitons in the $L^2$-critical framework, which cannot be seen in other critical equations such as $L^2$-critical NLS and $L^2$-critical generalized KdV. This property, of course, comes from the rich structure of a two-parameter family of solitons. 

We now state the solitons of \eqref{DNLSb} in more detail. The equation~\eqref{DNLSb} admits a two-parameter family of solitons
\[u_{\omega,c}(t,x)
  =e^{i\omega t}\phi_{\omega,c}(x-ct), \]
where $(\omega,c)\in\R^2$ satisfies
\begin{gather}\label{eq:1.2}
  \left\{\begin{alignedat}{2}
  &\mathopen{}-2\sqrt{\omega}<c\le 2\sqrt{\omega}\quad
  &&\text{if}~b>-3/16,
\\&\mathopen{}-2\sqrt{\omega}
  <c<-2\kappa_*\sqrt{\omega}\quad
  &&\text{if}~ b\le-3/16,
  \end{alignedat}\right.
\\\notag
  \kappa_*
  =\kappa_*(b)
  \ce\sqrt{\frac{-\gamma}{1-\gamma}}
  =\sqrt{\frac{3+16b}{16b}}\in(0,1)\quad
  \text{when}~b\le -3/16,
\\\gamma
  =\gamma(b)
  \ce 1+\frac{16}{3}b,
 \notag
\end{gather} 
and $\phi_{\omega,c}$ is explicitly written as
\begin{gather*} 
  \phi_{\omega,c}(x)
  =\Phi_{\omega,c}(x)\exp\left(i\frac{c}{2}x
  -\frac{i}{4}\int_{-\infty}^x\Phi_{\omega,c}(y)^2\,dy\right),
\\\notag
  \Phi_{\omega,c}(x)
  =\left\{\begin{alignedat}{2}
  &\left(\frac{2(4\omega-c^2)}{\sqrt{c^2+\gamma(4\omega-c^2)}\cosh(\sqrt{4\omega-c^2}\,x)-c}\right)^{1/2}\quad
  &&\text{if $-2\sqrt{\omega}<c<2\sqrt{\omega}$}, 
\\&\left(\frac{4c}{(cx)^2+\gamma}\right)^{1/2}\quad
  &&\text{if $c=2\sqrt{\omega}$}.
\end{alignedat}\right.
\end{gather*} 
We note that $\phi_{\omega,c}\in H^1(\R)$ is the nontrivial solution of the stationary equation
\begin{equation}\label{eq:1.3}
  -\phi''
  +\omega\phi
  +ci\phi'
  -i|\phi|^2\phi'
  -b|\phi|^4\phi
  =0,\quad 
  x\in\R,
\end{equation}
and that $\Phi_{\omega,c}$ is the positive even solution of
\begin{equation}\label{eq:1.4}
  -\Phi''
  +\Bigl(\omega-\frac{c^2}{4}\Bigr)\Phi
  +\frac{c}{2}|\Phi|^2\Phi
  -\frac{3}{16}\gamma|\Phi|^4\Phi
  =0,\quad
  x\in\R.
\end{equation}
The equation \eqref{eq:1.4} has nontrivial $H^1$-solutions if and only if $(\omega,c)$ satisfies \eqref{eq:1.2}. 

For $(\omega ,c)$ satisfying \eqref{eq:1.2}, one can rewrite
$(\omega ,c) =(\omega ,2\kappa\sqrt{\omega})$, where the parameter $\kappa$ satisfies 
\begin{align*}
\begin{array}{ll}
\ds -1 <\kappa\leq 1&\ds\text{if}~b>-3/16,\\[7pt]
\ds -1 <\kappa<-\kappa_{\ast}&\ds\text{if}~b\leq -3/16. 
\end{array}
\end{align*}
For each parameter $\kappa$, the following curve
\begin{align*}
\R^+ \ni\omega \mapsto (\omega , 2\kappa\sqrt{\omega}) \in \R^2 
\end{align*}
gives the scaling of the soliton:
\begin{align*}
\phi_{\omega ,2\kappa\sqrt{\omega}}(x) =\omega^{1/4} \phi_{1,2\kappa} (\sqrt{\omega}x)\quad\text{for}~x\in\R.
\end{align*}
When $b\ge 0$, there exists a unique $\kappa_0=\kappa_0(b)\in (0,1]$ such that
\begin{align*}
E(\phi_{1,2\kappa_0})=P(\phi_{1,2\kappa_0})=0,
\end{align*}
which implies that the soliton $u_{\omega,2\kappa_0\sqrt{\omega}}$ corresponds to the degenerate case.\footnote{See \eqref{eq:1.7} below more precisely.} 
We note that $0<\kappa(b)<1$ if $b>0$, and $\kappa_0(0)=1$. Therefore, algebraic solitons of \eqref{DNLS} correspond to the degenerate case, while degenerate solitons for $b>0$ have exponential decay at space infinity. 

The degenerate soliton can be also found in a different context, for example, the $L^2$-critical NLS
\begin{equation}
\tag{NLS}
\label{NLS}
  iu_t
  =- u_{xx}
  -|u|^{4}u,\quad
  (t,x)\in\R\x\R, 
\end{equation} 
and the $L^2$-critical generalized KdV equation
\begin{equation}
\tag{gKdV}
\label{eq:gKdV}
  u_t=-(u_{xx}
  +u^5)_x,\quad
  (t,x)\in\R\x\R. 
\end{equation}
The equations \eqref{NLS} and \eqref{eq:gKdV} have the same conserved quantities:
\begin{align}
\tag{Energy}
&{\cE}(v)= \frac{1}{2}\| v_x\|_{L^2}^2  -\frac{1}{6} \| v\|_{L^6}^6 ,\\
\tag{Mass}
&{\cM}(v) =\| v\|_{L^2}^2. 
\end{align}
\eqref{NLS} has the standing wave $e^{it}Q(x)$ and \eqref{eq:gKdV} has the traveling wave $Q(\cdot -t)$, where $Q(x)=\frac{3^{1/4}}{\cosh^{1/2}(2x)}$ is the positive even solution of
\begin{align*}
-Q'' +Q -Q^5=0, \quad x\in \R,
\end{align*}
and $Q$ is an optimizer of the following Gagliardo--Nirenberg inequality (see \cite{W82}):
\begin{align}
\label{GN1}
\frac{1}{6}\| f\|_{L^6}^6 \leq \frac{1}{2}\l( \frac{\cM(f)}{\cM(Q)} \r)^2\| f_x\|_{L^2}^2
\quad\text{for}~f\in H^1(\R).
\end{align}
In particular $\cE(Q)=0$ holds, which implies that the solitons $e^{it}Q(x)$ and $Q(\cdot -t)$ correspond to the degenerate case. It is also known that these degenerate solitons are unstable (see \cite{W82, MM01}).

Instability of degenerate solitons is important to understand the global dynamics of \eqref{NLS} and \eqref{eq:gKdV}. It follows from \eqref{GN1} and conservation laws that if the initial data $u_0\in H^1(\R)$ of \eqref{NLS} or \eqref{eq:gKdV} satisfies $\cM(u_0)<\cM (Q)$, the corresponding $H^1$-solution is global and satisfies
\begin{align*}
\frac{1}{2}\l( 1- \l( \frac{\cM(u_0)}{\cM(Q)} \r)^2\r) \| u_x(t)\|_{L^2}^2 \leq \cE(u_0)
\quad\text{for all}~ t\in\R.
\end{align*}
For \eqref{NLS}, it is known that finite time blow-up occurs for the initial data satisfying $\cM(u_0)>2\pi$ and $\cE(u_0)<0$ (see \cite{OT91}). On the other hand, for \eqref{eq:gKdV} existence of blow-up solutions is a more delicate problem. Martel and Merle \cite{MM02} proved that finite time blow-up occurs for the initial data satisfying
\begin{align}
\label{eq:1.6}
  \cE(u_0)<0,\quad
  \cM(Q)<\cM(u_0)<\cM(Q)+\alpha_0
\end{align}
and some decay condition, where $\alpha_0>0$ is a small constant. We note that before the work \cite{MM02}, the same authors \cite{MM01} proved instability of the soliton in a qualitative way, which led to an important step for proving the existence of blow-up solutions.


For \eqref{DNLSb} in the case $b\ge 0$,\footnote{The global result for the case $b<0$ was also established in \cite{H19}.} it was proved in \cite{Wu15, H19} that if the initial data $u_0\in H^1(\R)$ satisfies $M(u_0)<M(\phi_{1,2\kappa_0})\ec M^*$, then the corresponding $H^1$-solution is global and satisfies
\begin{align*}
\|u_x(t) \|_{L^2} \leq C(\|u_0 \|_{H^1})\quad \text{for all}~t\in\R,
\end{align*}
where the constant in the right-hand side is composed of the conserved quantities $E(u_0)$, $M(u_0)$, and $P(u_0)$. 
For \eqref{DNLS} this mass condition is nothing but the $4\pi$-mass condition.
In the recent progress of studies on \eqref{DNLS}, global well-posedness without the smallness assumption of the mass was established by taking advantage of completely integrable structure (see \cite{PSS17, JLPS20, BP, HKV}). These results give a remarkable difference with other $L^2$-critical equations \eqref{NLS} and \eqref{eq:gKdV}, while the dynamics of \eqref{DNLS} in the energy space is not yet clear including the fundamental problem of stability/instability of algebraic solitons. 



It was proved in \cite{H19} that the mass threshold $M^*$ gives a certain turning point in variational properties of \eqref{DNLSb}. This suggests that global dynamics of \eqref{DNLSb} will change at the mass of $M^*$. From the variational point of view, $M^*$ corresponds to the mass threshold $\cM(Q)$ in \eqref{NLS} and \eqref{eq:gKdV}. Therefore, to make clear the dynamics around the mass of $M^*$ is an important step towards understanding the global dynamics of \eqref{DNLSb}. To this end, in this paper we study instability properties of degenerate solitons of \eqref{DNLSb} in a qualitative way. 

We first give a precise definition of stability and instability of solitons.
\begin{definition}\label{def:1.1}
We say that the soliton $u_{\omega,c}$ of \eqref{DNLSb} is \emph{stable} if for any $\alpha>0$ there exists $\beta>0$ such that if $u_0\in H^1(\R)$ satisfies $\|u_0-\phi_{\omega,c}\|_{H^1}<\beta$, the solution $u(t)$ of \eqref{DNLSb} exists globally in time and satisfies
\[
\sup_{t\in\R}\inf_{(\theta,y)\in\R^2}\|u(t)-e^{i\theta}\phi_{\omega,c}(\cdot-y)\|_{H^1}<\alpha.
\]
Otherwise, we say that the soliton $u_{\omega,c}$ is \emph{unstable}. 
\end{definition}

We now review the known stability results related to our work. When $b=0$, Colin and Ohta \cite{CO06} proved by applying variational approach that if $\omega>c^2/4$, the soliton $u_{\omega,c}$ is stable. For the case $c=2\sqrt{\omega}$ some kinds of stability properties were studied in \cite{KPR06, KW18}, while the stability or instability in the sense of Definition \ref{def:1.1} remains an open problem.
Liu, Simpson, and Sulem \cite{LSS13b} calculated linearized operators of the generalized derivative nonlinear Schr\"{o}dinger equation
\begin{align}
\tag{gDNLS}
\label{GD}
iu_t+u_{xx}+i|u|^{2\sigma}u_x=0,\quad (t,x) \in \R \times \R, ~\sigma >0,
\end{align}
and studied stability of nondegenerate solitons by applying the abstract theory of Grillakis, Shatah, and Strauss \cite{GSS87, GSS90} (see also \cite{GW95} for partial results in this direction). Although well-posedness in the energy space for \eqref{GD} was assumed in \cite{LSS13b}, the well-posedness problem was later dealt with in \cite{Sa15, HO16, LPS19}.

When $b>0$, Ohta~\cite{O14} proved by applying variational approach in \cite{SS85, G91, CO06} that the soliton $u_{\omega,c}$ is stable if $-2\sqrt{\omega}<c<2\kappa_0\sqrt{\omega}$, and unstable if $2\kappa_0\sqrt{\omega}<c<2\sqrt{\omega}$. Ning, Ohta, and Wu \cite{NOW17} proved that the algebraic soliton is unstable for small $b>0$, where the assumption of smallness is used for construction of the unstable direction. We note that the momentum of the soliton $P(\phi_{\omega,c})$ is positive in the stable region $\{-2\sqrt{\omega}<c<2\kappa_0\sqrt{\omega}\}$, negative in the unstable region $\{2\kappa_0\sqrt{\omega}<c\le 2\sqrt{\omega}\}$, and zero on $\{c=2\kappa_0\sqrt{\omega}\}$ (see Remarks~2 and 3 of \cite{O14}). 
When $b<0$ the second author \cite{H20} proved by developing variational approaches in \cite{CL82, S83, CO06, O14} that all solitons including algebraic solitons are stable. We note that if $b<0$, the momentum of all solitons is positive. 

It is known that the stability/instability depends on the spectral properties of the Hessian matrix of the two-variable function
\[d(\omega,c)
  \ce S_{\omega,c}(\phi_{\omega,c}), \]
where $S_{\omega,c}$ is the action defined by
\[S_{\omega,c}(v)
  \ce E(v)
  +\frac{\omega}{2}M(v)
  +\frac{c}{2}P(v). \]
The abstract theory of Grillakis, Shatah, and Strauss~\cite{GSS87,GSS90} implies that under the spectral assumptions, which are verified for $\om>c^2/4$ in Proposition~\ref{prop:1.2} below, the soliton $u_{\om,c}$ is stable if $d''(\omega,c)$ has a positive eigenvalue, and unstable if $d''(\omega,c)$ has two negative eigenvalues. From a direct computation, we have the identity
\begin{equation}
\label{eq:1.7}
  \det[d''(\omega,c)]
  =\frac{-2P(\phi_{\omega,c})}{\sqrt{4\omega-c^2}\{c^2+\gamma(4\omega-c^2)\}}\quad
  \text{for}~\omega>\frac{c^2}{4}.
\end{equation}
This identity shows that the number of the positive/negative eigenvalues of $d''(\omega,c)$ depends on the sign of $P(\phi_{\omega,c})$. In particular, if $P(\phi_{\omega,c})=0$, then $d''(\omega,c)$ has a zero eigenvalue, which corresponds to the degenerate case.

We note that the abstract theory in \cite{GSS87, GSS90} is not applicable to degenerate solitons. In \cite{CP03, O11, M12} instability of degenerate solitons with one-parameter is studied in the abstract framework. The first author \cite{F17} extended the work of \cite{O11} to degenerate solitons with two-parameter. However, these results are not applicable to degenerate solitons of $L^2$-critical equations \eqref{NLS}, \eqref{eq:gKdV} and \eqref{DNLSb}. 
Recently, Ning~\cite{N20} proved the instability of the  soliton $u_{\omega,2\kappa_0\sqrt{\omega}}$ of \eqref{DNLSb} for sufficiently small $b>0$. The proof was done by combining localized virial identities and modulation analysis, whose argument was originally developed in \cite{WuKG, GNW}.

Our approach in the present paper is motivated by the works \cite{MM01, FHR19} on instability of degenerate solitons of \eqref{eq:gKdV}. 


We now state our results of this paper. We first organize the spectral properties of the linearized operator around the soliton. The linearized operator is explicitly written as
\begin{align}
\label{eq:1.8}
  L_{\omega,c}v
  \ce{}&S_{\omega,c}''(\phi_{\omega,c})v
\\\notag
  ={}&\mathopen{}-v_{xx}
  +\omega v
  +civ_{x}-i|\phi_{\omega,c}|^2v_{x}
  -2i\Re(\phi_{\omega,c}\ol{v})\phi_{\omega,c}'
\\&\qquad
  -b|\phi_{\omega,c}|^4v
  -4b|\phi_{\omega,c}|^2\Re(\phi_{\omega,c}\ol{v})\phi_{\omega,c}
\notag
\end{align}  
for $v\in H^1(\R)$. The following claim is used as a basic tool in the proof of our main result.
\begin{proposition}\label{prop:1.2}
Let $b\in\R$ and let $(\omega,c)$ satisfy \eqref{eq:1.2}. Then the space $H^1(\R)$ is decomposed as the orthogonal direct sum
\[H^1(\R)
  =\mathcal{N}_{\omega,c}
  \oplus \mathcal{Z}_{\omega,c}
  \oplus \mathcal{P}_{\omega,c}. \]
Here $\mathcal{N}_{\omega,c}$ is the negative subspace of $L_{\omega,c}$ spanned by the eigenvector $\chi_{\omega,c}$ corresponding to the simple negative eigenvalue $\lambda_{\omega,c}$, $\mathcal{Z}_{\omega,c}$ is the kernel of $L_{\omega,c}$ spanned by $i\phi_{\omega,c}$ and $\phi_{\omega,c}'$, and $\mathcal{P}_{\omega,c}$ is the positive subspace of $L_{\omega,c}$ such that 
\begin{enumerate}[\rm(i)]
\item if $-2\sqrt{\omega}<c<2\sqrt{\omega}$, then there exists a positive constant $k>0$ such that for any $p\in\mathcal{P}_{\omega,c}$
\begin{align}
\label{eq:1.9}
\langle L_{\omega,c}p,p\rangle
  \ge k\|p\|_{H^1}^2, 
\end{align}

\item if $c=2\sqrt{\omega}$, then for any $p\in\mathcal{P}_{\omega,c}\setminus\{0\}$
\begin{align}
\label{eq:1.10}
  \langle L_{\omega,c}p,p\rangle
  >0.
\end{align}
\end{enumerate}
\end{proposition}
We prove Proposition~\ref{prop:1.2} by mainly following the argument 
in \cite{LSS13b}. Here we treat the case $c=2\sqrt{\omega}$, which was not considered in previous works. As in the assertion (ii), the coercivity fails for the case $c=2\sqrt{\omega}$ because the essential spectral of $L_{\omega,c}$ consists of the interval $[0,\infty)$ for this case. 

We now state our main result, which concerns the instability of degenerate solitons of \eqref{DNLSb}.
\begin{theorem}\label{thm:1.3}
Let $b>0$ and $c=2\kappa_0\sqrt{\omega}$ and let $\chi_{\omega,c}$ be as in Proposition~\ref{prop:1.2}. Then there exist $\alpha,\beta\in(0,1)$ such that 
if $\ve_0\ce u_0-\phi_{\omega,c}$ for the initial data $u_0\in H^1(\R)$ satisfies
\begin{align}
\label{eq:1.11}
  0<\|\ve_0\|_{H^1}^2
  \le\beta|(\ve_0,\phi_{\omega,c})_{L^2}|,\quad
  \ve_0\perp\{\chi_{\omega,c},i\phi_{\omega,c},\phi_{\omega,c}',i\phi_{\omega,c}'\},  
\end{align}
then there exists $t_0=t_0(u_0)\in\R$ such that the solution $u(t)$ of \eqref{DNLSb} satisfies
\begin{align*}
\inf_{(\theta,y)\in\R^2} \| u(t_0)-e^{i\theta}\phi_{\omega,c}(\cdot -y) \|_{H^1} \ge \alpha.
\end{align*}
In particular, the soliton $u_{\omega,c}$ is unstable. 
\end{theorem}
\begin{remark}
We can construct $\ve_0$ satisfying \eqref{eq:1.11} as follows. One can easily show that the functions $\chi_{\omega,c}$, $i\phi_{\omega,c},\phi_{\omega,c}'$, $\phi_{\omega,c}$, $i\phi_{\omega,c}'$ are linearly independent. Applying the Gram--Schmidt process, we have a function $\ve_1\in H^1(\R)$ satisfying 
\[(\ve_1,\phi_{\omega,c})\ne0,\quad
  \ve_1\perp\{\chi_{\omega,c},i\phi_{\omega,c},\phi_{\omega,c}',i\phi_{\omega,c}'\}. \]
Then $\ve_0\ce \delta \ve_1$ for small $\delta>0$ satisfies \eqref{eq:1.11}.
\end{remark}
\begin{remark}
If we replace the assumption \eqref{eq:1.11} by 
\begin{align*}
0<\|\eps_0\|_{H^1}^2
  \le \beta\l|(\ve_0,i\phi_{\omega,c}')_{L^2}\r|,\quad 
\ve_0\perp\{\chi_{\omega,c},i\phi_{\omega,c},\phi_{\omega,c}',\phi_{\omega,c}\},
\end{align*}
then the conclusion in Theorem \ref{thm:1.3} still holds.
\end{remark}
\begin{remark}
\label{rem:1.6}
In \cite{N20} some explicit function was used as a negative direction of $L_{\omega,c}$ instead of the eigenfunction $\chi_{\omega,c}$. The smallness assumption on $b>0$ in \cite{N20} comes from the construction of a negative direction and the explicit formula is also used for the control of modulation parameters. Although one cannot expect the explicit formula of  $\chi_{\omega,c}$,\footnote{In contexts of \eqref{NLS} and \eqref{eq:gKdV} one can use the explicit eigenfunction for negative eigenvalue of the linearized operator.}
we construct and control modulation parameters by using the scaling properties of the equation. Moreover, we obtain a large set of initial data yielding the instability while in \cite{N20} the only one unstable direction is found.
\end{remark}


For the proof of Theorem~\ref{thm:1.3} we use modulation theory and the virial identity
\begin{align}
\label{eq:1.12}
\frac{d}{dt}\Im\int xu_x(t,x)\ol{u}(t,x) &=4E(u_0),
\end{align}
but we avoid a direct use of this identity. We consider the decomposition
\begin{align}
\label{eq:1.13}
u(t,x)=\frac{e^{i\theta(t)}}{\lambda(t)^{1/2}}\l(\phi_{\omega,c}+\eps \r)\l(t,\frac{x-x(t)}{\lambda(t)} \r),
\end{align}
where $\lambda(t)>0$, $\theta(t)\in \R$, $x(t)\in \R$, and the function $\eps(t,x)$ satisfies suitable orthogonal conditions (see Proposition \ref{prop:3.2}). If we put the formula \eqref{eq:1.13} into \eqref{eq:1.12}, the left-hand side of \eqref{eq:1.12} yields the quantity
\begin{align}
\label{eq:1.14}
  \frac{d}{dt}\Im\int \eps(t,x)\Lambda\phi_{\omega,c}(x)
  \quad \l(\Lambda f\ce\tfrac{f}{2}+xf_x\r),
\end{align}
which plays an essential role in our proof. This quantity has already been effectively used on the studies of the blow-up dynamics of \eqref{NLS} (see, e.g., \cite{MR05}), but it seems to be new in the contexts of \eqref{DNLSb}, \eqref{DNLS} and \eqref{GD}. The quantity \eqref{eq:1.14} is well-defined in the $H^1$-setting, so we do not need any cut-off arguments, which becomes a much simpler argument than previous works \cite{WuKG, GNW, N20}. Moreover, our proof gives a close relation to instability theory on \eqref{eq:gKdV} (see Appendix \ref{sec:A}).

The arguments except one step (Lemma~\ref{lem:3.7}) in our proof work for the case $b=0$ and $c=2\sqrt{\omega}$, i.e., algebraic solitons of \eqref{DNLS}, in exactly the same way.\footnote{For the proof of Lemma \ref{lem:3.7}, we use the coercivity property of $L_{\omega,c}$ which does not hold in the case $c=2\sqrt{\omega}$.}
Although we could not complete the proof of Theorem \ref{thm:1.3} for the case $b=0$, unstable directions are detected in the same way as the case $b>0$ (see Lemma \ref{lem:4.3}). 
Therefore, we believe that the conclusion of Theorem \ref{thm:1.3} is still true for algebraic solitons of \eqref{DNLS}.

In the assumption of Theorem~\ref{thm:1.3}, if we consider the initial data $u_0=\phi_{\omega,c}+\ve_0$ with $(\ve_0, \phi_{\omega,c})_{L^2}>0$, 
then 
\begin{equation}
\label{eq:1.15}
  E(u_0)<0,\quad 
  M(\phi_{\omega,c})
  <M(u_0)
  <M(\phi_{\omega,c})+\beta_0,
\end{equation} 
where $\beta_0$ is a small constant. We note that the condition \eqref{eq:1.15} corresponds to the blow-up set of \eqref{NLS} and \eqref{eq:gKdV}, and so Theorem \ref{thm:1.3} gives an important clue to construct a singular solution of \eqref{DNLSb}.

The rest of this paper is organized as follows. In Section \ref{sec:2} we study the spectra of the linearized operator $L_{\omega,c}$ and prove Proposition~\ref{prop:1.2}. In Section \ref{sec:3} we construct the modulation parameters satisfying suitable orthogonal conditions and control these parameters. In Section \ref{sec:4} we organize the virial identities. In Section \ref{sec:5} we complete the proof of Theorem \ref{thm:1.3} by using the estimates obtained in previous sections.

\section{Structure of the linearized operator}
\label{sec:2}
In this section, we study the structure of the linearized operator $L_{\omega,c}$. Throughout this section, we assume that $(\omega,c)$ satisfies \eqref{eq:1.2}.
For simplicity we often drop the subscript $(\omega,c)$ as
\[S=S_{\omega,c},\quad
  \phi=\phi_{\omega,c},\quad
  \Phi=\Phi_{\omega,c}. \]
We define the function $\eta_{\omega,c}$ as
\begin{align}
\label{eq:2.1}
\eta(x)
  =\eta_{\omega,c}(x)
  =\frac{c}{2}x-\frac14\int_{-\infty}^x\Phi_{\omega,c}(y)^2\,dy,
\end{align}
and define the operator $\tilde L_{\omega ,c}$ as
\[\tilde L
  =\tilde L_{\omega,c}
  =e^{-i\eta_{\omega,c}(x)}L_{\omega,c}e^{i\eta_{\omega,c}(x)}. \]
For $w\in H^1(\mb{R})$ we set $f=\Re w$ and $g=\Im w$.
After a direct calculation, $\tilde{L}w$ is explicitly represented as
\begin{align}
\label{eq:2.2}
  \tilde{L}w
  ={}&\mathopen{}
  -{w}_{xx}
  +\l(\omega -\frac{c^2}{4}\r)w +\frac{c}{2}\Phi^2w +c\Phi^2\Re w 
  -\frac{3}{16}\gamma\Phi^4w
  -\frac{3}{4}\gamma\Phi^4\Re w
\\\notag
  &+\frac{1}{4}\Phi^4\Re w -\frac{i}{2}\Phi^2w_{x}
  +\frac{i}{2}\Phi\Phi'w -2i\Phi\Phi'\Re w
\\\notag
  ={}&L_{11}f
  +L_{12}g
  +\frac{1}{4}\Phi^4f
  +i(L_{21}f
  +L_{22}g),
\end{align}
where
\begin{align*}
  L_{11}
  &\ce-\pt_x^2+U_{\Phi},
  &U_{\Phi}
  &\ce\left(\omega-\frac{c^2}{4}\right)
  +\frac32c\Phi^2
  -\frac{15}{16}\gamma\Phi^4,
\\L_{12} &\ce 
  \frac{1}{2}\Phi^2\del_x
  -\frac{1}{2}\Phi\Phi',
\\L_{21} &\ce 
  -\frac{1}{2}\Phi^2\del_x -\frac{3}{2}\Phi\Phi',
\\L_{22}
  &\ce-\del_x^2 +V_{\Phi},
  &V_{\Phi}
  &\ce\left(\omega -\frac{c^2}{4}\right)
  +\frac{c}{2}\Phi^2  
  -\frac{3}{16}\gamma\Phi^4.
\end{align*}
Since $e^{i\eta(x)}$ is a unitary operator, the spectral property of $\tilde{L}$ is the same as that of $L$. In what follows, we investigate the spectra of the operator $\tilde{L}$. 

We first note that $\tilde{L}$ can be considered as compact perturbation of 
the operator $-\del_x^2+(\omega-c^2/4)$. Therefore, by Weyl's theorem we deduce that
\[
\sigma_{\mr{ess}}(\tilde{L})=\sigma_{\mr{ess}}\Bigl(-\pt_x^2+\Bigl(\omega-\frac{c^2}4\Bigr)\Bigr)
=\Bigl[\omega-\frac{c^2}4,\infty\Bigr)
\]
and the spectrum of $\tilde{L}$ in $(-\infty,\omega-c^2/4)$ consists of isolated eigenvalues.

\subsection{Kernel}
In this subsection we prove the nondegeneracy of the kernel of $\tilde{L}$. Our proof depends on the argument in \cite{LW18}.
\begin{lemma}\label{lem:2.1}
The following statement is true.
\begin{enumerate}[\rm(i)]
\item \label{enumi:i}
$\ker L_{11}=\spn\{\Phi_{\om,c}'\}$,
\item \label{enumi:ii}
$\ker L_{22}=\spn\{\Phi_{\om,c}\}$.
\end{enumerate}
\end{lemma}

\begin{proof}
Since $\Phi$ is a solution of \eqref{eq:1.4}, we have $L_{22}\Phi=0$. By differentiating the equation \eqref{eq:1.4}, we also have $L_{11}\Phi'=0$. Hence we have
\begin{align*}
  \ker L_{11}\supset\spn\{\Phi'\},\quad
  \ker L_{22}\supset\spn\{\Phi\}.
\end{align*}
It now suffices to show $\ker L_{22}\subset\spn\{\Phi\}$ because one can show $\ker L_{11}\subset\spn\{\Phi'\}$ by the same argument. Let $g\in\ker L_{22}$. We consider the Wronskian of $\Phi$ and $g$:
\[W(x)
  \ce\Phi'(x)g(x)-\Phi(x)g'(x). \]
From $\Phi,g\in H^2(\R)$, we have $W(x)\to0$ as $|x|\to0$. Since $L_{22}\Phi=L_{22}g=0$, we obtain
\[W'(x)
  =\Phi''g-\Phi g''
  =V_{\Phi}\Phi g
  -\Phi V_{\Phi}g
  =0. \]
Thus, we deduce $W\equiv 0$, which implies that $\Phi$ and $g$ are linearly dependent. This completes the proof.
\end{proof}
\begin{lemma}
\label{lem:2.2}
The kernel $\tilde{L}_{\omega ,c}$ is determined by
\begin{align*}
\ker \tilde{L}_{\om,c}=\spn\l\{i\Phi_{\om,c},\Phi_{\om,c}'-\frac{i}{4}\Phi_{\om,c}^3\r\},
\end{align*}
which is equivalent to $\ker L_{\om,c}=\spn\{i\phi_{\om,c},\phi_{\om,c}'\}$.
\end{lemma}
\begin{proof}
First we show $\ker \tilde{L}\supset\on{span}\l\{i\Phi,\Phi'-\frac{i}{4}\Phi^3\r\}$.
Since $\phi$ is a solution of \eqref{eq:1.3}, and the equation has symmetries under the phase and spatial translation, we have $S'(e^{i\theta}\phi(\cdot-y))=0$ for all $(\theta,y)\in\R\times\R$.
Differentiating this with respect to $\theta$ or $y$ at $(\theta,y)=0$, we have
\begin{equation}\label{eq:2.3}
Li\phi=0,\quad
L\phi'=0,
\end{equation}
respectively.
Since $e^{-i\eta(x)}L=\tilde{L}e^{-i\eta(x)}$ and $\phi=e^{i\eta(x)}\Phi$,
\eqref{eq:2.3} is equivalent to 
\[
\tilde{L}i\Phi=0,\quad
\tilde{L}\Big(\Phi'+i\frac{c}{2}\Phi-\frac{i}{4}\Phi^3\Big)=0.
\]
This implies $\ker \tilde{L}\supset\on{span}\l\{i\Phi,\Phi'-\frac{i}{4}\Phi^3\r\}$.

Next we show the inverse inclusion.
Let ${w}\in\ker \tilde{L}$, $f=\Re{w}$, and $g=\Im{w}$.
The expression \eqref{eq:2.2} of $\tilde{L}$ implies that $(f,g)$ satisfies the following system of ordinary differential equations:
\begin{align}
\begin{cases}
\displaystyle L_{11}f+L_{12}g+\frac{1}{4}\Phi^4 f =0, \\[8pt]
L_{21}f+L_{22}g=0.
\end{cases}
\label{eq:2.4}
\end{align}

Now we apply the following transformation to $g$:
\begin{align}
\label{eq:2.5}
g= h-\frac{1}{2}\Phi\int_{-\infty}^{x} \Phi f\,dy.
\end{align}
Then we have
\begin{align}
L_{12}g
+\frac14\Phi^4f
\label{eq:2.6}
&=\frac12\Phi^2g_x-\frac12\Phi\Phi'g
+\frac14\Phi^4f  \\
&=\frac12\Phi^2h_x
-\frac12\Phi\Phi'h.
\notag
\end{align}
Moreover, noting that
\begin{align*}
  \pt_x^2\left(\frac12\Phi\int_{-\infty}^x\Phi f\,dy\right)
  &=\frac12\Phi''\int_{-\infty}^x\Phi f\,dy
  +\frac{3}{2}\Phi\Phi'f
  +\frac{1}{2}\Phi^2f_{x}
\\&=\frac12\Phi''\int_{-\infty}^x\Phi f\,dy
  -L_{21}f,
\end{align*}
it follows from $L_{22}\Phi=0$ that
\begin{align} \label{eq:2.7}
L_{21}f+L_{22}g
&=L_{21}f
+L_{22}h
+\pt_x^2\left(\frac12\Phi\int_{-\infty}^x\Phi f\,dy\right)
-\frac12V_{\Phi}\Phi\int_{-\infty}^x\Phi f\,dy  \\
&=L_{22}h-\frac12(-\Phi''+V_{\Phi}\Phi)\int_{-\infty}^x\Phi f\,dy 
=L_{22}h.
\notag
\end{align}
Using \eqref{eq:2.6} and \eqref{eq:2.7} we write the equation \eqref{eq:2.4} as
\begin{align}\label{eq:2.8}
\left\{\begin{aligned}
  &L_{11}f+\frac{1}{2}\Phi\l( \Phi h_x-\Phi'h\r) =0, 
\\&L_{22}h=0.
\end{aligned}\right.
\end{align}
From the second equation in \eqref{eq:2.8} and Lemma~\ref{lem:2.1}~\eqref{enumi:ii}, we have $h=\alpha\Phi$ for some $\alpha\in\mb{R}$.
Substituting this into the first equation in \eqref{eq:2.8}, we get $L_{11}f=0$.
Therefore, Lemma~\ref{lem:2.1}~\eqref{enumi:i} implies that $f=\beta\Phi'$ for some $\beta\in\mb{R}$.
Substituting $h=\alpha\Phi$ and $f=\beta\Phi'$ into \eqref{eq:2.5}, we have
\begin{align*}
g
=\alpha\Phi
-\frac{\beta}{2}\Phi\int_{-\infty}^x\Phi\Phi'\,dy
=\alpha\Phi
-\frac{\beta}{4}\Phi\int_{-\infty}^x(\Phi^2)'\,dy
=\alpha\Phi
-\frac{\beta}{4}\Phi^3.
\end{align*}
Therefore, we obtain that
\begin{align*}
  w&=f+ig
  =\beta\Phi' +i\l(\alpha\Phi -\frac{\beta}{4}\Phi^3 \r) 
\\&=\alpha i\Phi +\beta\l( \Phi' -\frac{i}{4}\Phi^3\r)
  \in\spn\l\{ i\Phi, \Phi' -\frac{i}{4}\Phi^3 \r\}.
\end{align*}
This completes the proof.
\end{proof}

\subsection{Construction of a negative direction}
In this subsection we prove that $\tilde{L}_{\omega ,c}$ has exactly one negative eigenvalue. Our proof depends on the argument in \cite{LSS13b} (see also \cite{GW95}). The following expression of the quadratic form is useful to construct a negative direction.
\begin{lemma}
\label{lem:2.3}
Let ${w} \in H^1(\R)$, $f=\Re{w}$, and $g=\Im {w}$. Then we have
\begin{align}
\label{eq:2.9}
\langle \tilde{L}_{\om,c}{w},{w}\rangle
=\langle L_{11}f,f\rangle 
+ \frac{1}{4}\lVert\Phi_{\om,c}^2 f+2\Phi_{\om,c}\del_x(\Phi_{\om,c}^{-1}g)\rVert_{L^2}^2.
\end{align}
\end{lemma}

\begin{proof}
First, by the expression~\eqref{eq:2.2}, we have
\begin{align*}
\tbra[\tilde{L}{w} ,{w}] =\tbra[L_{11}f,g] +\tbra[L_{12}g,f] 
+\frac{1}{4}\tbra[ \Phi^4 f,f]+\tbra[L_{21}f,g] +\tbra[L_{22}g,g].
\end{align*}
We set $\tilde{g}=\Phi^{-1}g$. It follows from $L_{22}\Phi=0$ that
\begin{align*}
\tbra[L_{22}g,g]
&=\tbra[\tilde{g}(-\pt_x^2+V_\Phi)\Phi ,\Phi\tilde{g}] 
-\tbra[2\Phi'\tilde{g}_{x}+\Phi\tilde{g}_{xx} ,\Phi\tilde{g}]\\
&=-\tbra[\del_x(\Phi^2\tilde{g}_{x}) ,\tilde{g}]
=\|\Phi\tilde{g}_{x}\|_{L^2}^2.
\notag
\end{align*}
Next, we calculate the interaction terms as
\begin{equation*}
\tbra[L_{12}g ,f] 
=\Big\langle\frac{1}{2}\Phi^2\del_x(\Phi\tilde g)-\frac{1}{2}\Phi^2\Phi'\tilde g,f\Big\rangle
=\frac12\langle\Phi^3,f\tilde{g}_{x}\rangle
\end{equation*}
and
\begin{align*}
  \tbra[L_{21}f, g] 
  &=-\bigg\langle\frac{1}{2}\Phi^2f_{x}
  +\frac{3}{2}\Phi\Phi' f ,\Phi\tilde g\bigg\rangle
\\&=-\frac12\langle\Phi^3,\tilde gf_{x}\rangle
  -\frac12\langle\del_x(\Phi^3),f\tilde g\rangle
  =\frac12\langle\Phi^3,f\tilde{g}_{x}\rangle.
\notag
\end{align*}
Therefore we deduce that
\begin{align*}
\tbra[\tilde{L}{w} ,{w}] 
&=\tbra[L_{11}f,g]
+\frac{1}{4}\tbra[ \Phi^4 f,f]
+\langle\Phi^3,f\tilde{g}_{x}\rangle
+\|\Phi\tilde{g}_{x}\|_{L^2}^2 \\
&=\tbra[L_{11}f,f] 
+ \frac{1}{4}\|\Phi^2 f+2\Phi\tilde{g}_{x}\|_{L^2}^2.
\end{align*}
This completes the proof.
\end{proof}

\begin{lemma} \label{lem:2.4}
The operator $L_{11}$ has exactly one negative eigenvalue.
\end{lemma}

\begin{proof}
We note that $L_{11}$ is a compact perturbation of the operator $-\pt_x^2+(\omega-c^2/4)$. Therefore, by Weyl's theorem we deduce that
\[
\sigma_{\mr{ess}}(L_{11})=\sigma_{\mr{ess}}\Bigl(-\pt_x^2+\Bigl(\omega-\frac{c^2}4\Bigr)\Bigr)
=\Big[\omega-\frac{c^2}4,\infty\Big),
\]
and the spectrum of $L_{11}$ in $(-\infty,\omega-c^2/4)$ consists of isolated eigenvalues.
We note that $L_{11}\Phi'=0$ and that $\Phi'$ has exactly one zero point. By Sturm--Liouville theory we deduce that zero is the second eigenvalue of $L_{11}$, and that $L_{11}$ has one negative eigenvalue. Moreover, one can prove that the negative eigenvalue is simple (see, e.g., \cite[Theorem~B.59]{A09}).
This completes the proof.
\end{proof}

We denote the negative eigenvalue of $L_{11}$ in Lemma~\ref{lem:2.4} by $\lambda_{11}$ and its normalized eigenvector by $\chi_{11}$, that is,
\begin{equation} \label{eq:2.10}
L_{11}\chi_{11}
=\lambda_{11}\chi_{11}, \quad
\|\chi_{11}\|_{L^2}=1.
\end{equation}

\begin{lemma} \label{lem:2.5}
The operator $\tilde{L}_{\om,c}$ has exactly one negative eigenvalue.
\end{lemma}

\begin{proof}
Let 
\[\chi_{12}
  \ce-\frac12\Phi\int_{-\infty}^x\Phi\chi_{11}\,dy. \]
Then we have
\[
\Phi\pt_x(\Phi^{-1}\chi_{12})
=-\frac12\Phi^2\chi_{11}.
\]
Therefore, it follows from \eqref{eq:2.9} and \eqref{eq:2.10} that 
$
\chi_{*}\ce
\chi_{11}
+i\chi_{12}
$
satisfies
\[
\langle \tilde{L}\chi_*,\chi_*\rangle
=\langle L_{11}\chi_{11},\chi_{11}\rangle
=\lambda_{11}<0.
\]
This means that the operator $\tilde{L}$ has at least one negative eigenvalue.

Now we show that $\tilde{L}$ has exactly one negative eigenvalue. Assume that $\tilde{L}$ has two negative eigenvalues (including repeats) $\lambda_1\le\lambda_2<0$ with eigenvectors $\chi_1$ and $\chi_2$ such that
\begin{align*}
\tilde{L}\chi_1=\lambda_1\chi_1,\quad
\tilde{L}\chi_2=\lambda_2\chi_2,\quad
\|\chi_1\|_{L^2}
=\|\chi_2\|_{L^2}
=1,\quad
(\chi_1,\chi_2)_{L^2}=0.
\end{align*}
We note that by the formula~\eqref{eq:2.9} and Lemma~\ref{lem:2.4}, $\langle \tilde{L}p,p\rangle\ge0$ for each $p\in H^1(\R)$ satisfying $(\Re p,\chi_{11})_{L^2}=0$. Thus, it follows from $\langle \tilde{L}\chi_2,\chi_2\rangle=\lambda_2<0$ that $(\Re\chi_2,\chi_{11})_{L^2}\ne0$. If we set
\[
\alpha=-\dfrac{(\Re\chi_1,\chi_{11})_{L^2}}{(\Re\chi_2,\chi_{11})_{L^2}},\quad
p_0=\chi_1+\alpha\chi_2,
\]
then we have $(\Re p_0,\chi_{11})_{L^2}=0$. Hence we deduce that $\langle \tilde{L}p_0,p_0\rangle\ge0$. 
On the other hand, by a direct calculation we obtain
\[
\langle \tilde{L}p_0,p_0\rangle
=\lambda_1+\alpha^2\lambda_2
<0,
\]
which yields a contradiction. This completes the proof.
\end{proof}
\begin{remark}
\label{rem:2.6}
When $b\ge 0$, by variational characterization of the solitons (see \cite{CO06, FHI17, H19}) one can prove that $L_{\omega ,c}$ has exactly one negative eigenvalue (see the argument of \cite{LW18}). Our approach based on the formula \eqref{eq:2.9} is more elementary and applicable to the case $b<0$ in a unified way.
\end{remark}

\subsection{Spectral decomposition}
We now complete the proof of Proposition~\ref{prop:1.2}.
\begin{proof}[Proof of Proposition~\ref{prop:1.2}]
By Lemma~\ref{lem:2.2} and Lemma \ref{lem:2.5}, we have the following decomposition
\begin{align}
\label{eq:2.11}
H^1(\mb{R})
=\spn\{\tilde\chi\}
\oplus\spn\l\{i\Phi,\Phi'-\frac{i}{4}\Phi^3\r\}
\oplus\tilde{\mc{P}},
\end{align}
where $\tilde\chi$ is the eigenvector of $\tilde{L}$ corresponding to its negative eigenvalue $\lambda$ and $\tilde{\mc{P}}$ is the nonnegative subspace of $\tilde{L}$.
Since $\tilde{L}=e^{-i\eta(x)}Le^{i\eta(x)}$,
\eqref{eq:2.11} is equivalent that
\begin{align}
\label{eq:2.12}
H^1(\mb{R})
=\mc{N}
\oplus \mc{Z}
\oplus \mc{P},
\end{align}
where $\mc{N}$ is spanned by the negative eigenvector $\chi\ce e^{i\eta(x)}\tilde\chi$ of $L$, $\mc{Z}\ce\spn\{i\phi,\phi'\}$ is its kernel, and $\mc{P}\ce e^{i\eta(x)}\tilde{\mc{P}}$ is its nonnegative subspace. The rest of the proof is to show the positivity of $L$ on $\mc{P}$. 

(i) We consider the case $-2\sqrt{\omega}<c<2\sqrt{\omega}$.
Since $\sigma_{\rm ess}(L)=[\omega-c^2/4,\infty)$, the spectra of $L$ except for its negative eigenvalue and zero eigenvalue are positive and bounded away from zero. 
Therefore, there exists a positive constant $\delta_0>0$ such that 
\begin{align}
\label{eq:2.13}
\langle Lp,p\rangle\ge \delta_0\|p\|_{L^2}^2
\quad\text{for all}~p\in\mc{P}.
\end{align}
From the explicit formula \eqref{eq:1.8}, there exists a positive constant $C_0$ such that
\begin{align*}
  \tbra[Lv,v]
  \ge\frac12\|v_x\|_{L^2}^2
  -C_0\|v\|_{L^2}^2
\end{align*}
for all $v\in H^1(\R)$. 
Combined with \eqref{eq:2.13}, we have 
\begin{align*}
  \|p\|_{H^1}^2
  \le 2\tbra[Lp,p]
  +(1+2C_0)\|p\|_{L^2}^2
  \le\l(2+\frac{1+2C_0}{\delta_0}\r)\tbra[Lp,p]
\end{align*}
for all $p\in\mc{P}$, which shows the desired inequality \eqref{eq:1.9}.

(ii) We now consider the case $c=2\sqrt{\omega}$.
 Assume by contradiction that there exists $p_0\in\mc{P}$ such that $\|p_0\|_{L^2}=1$ and $\langle Lp_0,p_0\rangle=0$.
Then we obtain the following relation:
\[\langle Lp_0,p_0\rangle
  =\min\{\langle Lp,p\rangle\colon\,
  \lVert p\rVert_{L^2}=1,\ 
  (\chi,p)_{L^2}
  =(i\phi,p)_{L^2}
  =(\phi',p)_{L^2}
  =0\}. \]
This minimization problem implies that there exist Lagrange multipliers $\alpha_1$, $\alpha_{2}$, $\alpha_{3}$, and $\alpha_4$ such that
\begin{align*}
Lp_0=\alpha_1\chi+\alpha_{2}i\phi+\alpha_{3}\phi'+\alpha_{4}p_0.
\end{align*}
By the orthogonal conditions and $\langle Lp_0,p_0\rangle=0$, we have $\alpha_{1}=\alpha_{2}=\alpha_3=\alpha_4=0$. Therefore, $p_0\in\ker L\cap\mc{P}=\{0\}$, which is a contradiction. Hence \eqref{eq:1.10} holds.
\end{proof}

\section{Modulation theory} 
\label{sec:3}

In this section we organize modulation theory for three fundamental symmetries which are phase, translation, and scaling. 

We prepare some notations. For $\alpha>0$ we define a tubular neighborhood around the soliton 
$\phi_{\omega,c}$ by
\begin{align*}
  U_{\alpha}=
  \{u\in H^1(\R)\colon\,
  \inf_{(\theta,z)\in\R^2}\|e^{i\theta}u(\cdot +z)-\phi_{\om,c}\|_{H^1}<\alpha\}.
\end{align*}
For $u\in H^1(\R)$, $\lambda>0$, and $\theta,y\in\R$, we denote the function $\eps$ by
\begin{align*}
  \eps(\lambda,\theta,x;u) 
  =\lambda^{1/2}e^{-i\theta}u(\lambda\cdot +x)-\phi_{\om,c}.
\end{align*}
For $\lambda>0$ and $f:\R\to\C$, we define the rescaling
\begin{equation*}
 f^\lambda (y)=\lambda^{1/2}f(\lambda y).
\end{equation*}
Let $\Lambda$ be the generator of this transformation as
\begin{equation*}
  \Lambda f
  \ce\del_{\lambda}f^\lambda|_{\lambda=1}
  =\frac{f}{2}
  +yf_y. 
\end{equation*}
We note that $\Lambda$ is skew-symmetric, i.e., 
\begin{align*}
  (\Lambda f ,g)_{L^2}
  =-(f,\Lambda g)_{L^2}. 
\end{align*}

\subsection{Construction of modulation parameters} \label{sec:3.1}
We construct the modulation parameters $\lambda$, $\theta$, and $x$ satisfying suitable orthogonal conditions. We first prepare the following lemma.
\begin{lemma}\label{lem:3.1}
Assume that $(\omega,c)$ satisfy $\eqref{eq:1.2}$. Then we have
\begin{enumerate}[\rm(i)]
\item $(\Lambda\phi_{\omega,c},i\phi_{\omega,c})_{L^2}=(\Lambda\phi_{\omega,c},\phi_{\omega,c}')_{L^2}
  =0. $
\end{enumerate}
If we further assume $b\ge 0$ and $c=2\kappa_0(b)\sqrt{\omega}$, then we have
\begin{enumerate}[\rm(i)]
\setcounter{enumi}{1}
\item 
$(i\phi_{\omega,c}',\Lambda\phi_{\omega,c})_{L^2}
  =(i\phi_{\omega,c}',\phi_{\omega,c})_{L^2}=0$, 
\item 
$(\Lambda\phi_{\omega,c},\chi_{\omega ,c})_{L^2}
  \neq 0$.
\end{enumerate}
\end{lemma}
\begin{proof}
(i) It follows from the explicit formula of $\eta$ (see \eqref{eq:2.1}) that
\begin{align*}
\eta'
  &=\frac{c}{2}-\frac{1}{4}\Phi^2,
\\
\phi'
  &=e^{i\eta}\l(i\eta'\Phi+\Phi'\r)
  =e^{i\eta}\l(i\frac{c}{2}\Phi-\frac{i}{4}\Phi^3+\Phi'\r).
\end{align*}
Since $\Phi$ is a real-valued and even function, one computes easily that
\begin{align*}
  &(\Lambda\phi,i\phi)_{L^2}
  =\l(\tfrac{\phi}{2}+y\phi',i\phi\r)_{L^2}
  =(y\phi',i\phi)_{L^2}
\\&\quad=\Re\int y\l(i\frac{c}{2}\Phi-\frac{i}{4}\Phi^3+\Phi'\r)\cdot(-i\Phi)
  =\Re\int y\l(\frac{c}{2}\Phi^2-\frac{1}{4}\Phi^4\r)
  =0,
\\&(\Lambda\phi,\phi')_{L^2}
  =\l(\tfrac{\phi}{2}+y\phi',\phi'\r)_{L^2}
\\&\quad=(y\phi',\phi')_{L^2}
  =\Re\int y\biggl\{(\Phi')^2+\biggl(\frac{c}{2}\Phi-\frac14\Phi^3\biggr)^2\biggr\}
  =0.
\end{align*}

\noindent
(ii) 
Since $P(\phi)=0$ by $c=2\kappa_0\sqrt{\omega}$, we have 
\begin{align*}
  (i\phi',\Lambda\phi)_{L^2}
  =(i\phi',\tfrac{\phi}{2}+y\phi')_{L^2}
  =\Re\int iy|\phi'|^2=0.
\end{align*}

\noindent
(iii) Suppose that $(\Lambda\phi,\chi)_{L^2}=0$. From (i) proved just above and Proposition~\ref{prop:1.2}, we obtain $\langle L\Lambda\phi,\Lambda\phi\rangle>0$. On the other hand, by twice differentiating the relation
\begin{align*}
  S(\phi^\lam)
  =\lam^2E(\phi)+\frac{\om}{2} M(\phi)+\frac{\lam c}{2}P(\phi)
\end{align*}
at $\lambda=1$, we have $\langle L\Lambda\phi,\Lambda\phi\rangle=2E(\phi)$. 
Since $E(\phi)=0$ by $c=2\kappa_0\sqrt{\omega}$, 
we deduce that $\langle L\Lambda\phi,\Lambda\phi\rangle=0$. This is a contradiction. 
\end{proof}

The next proposition is the foundation of the modulation analysis.
\begin{proposition}\label{prop:3.2}
Let $b\ge0$ and $c=2\kappa_0\sqrt{\omega}$. Then there exist constants $\alpha_0>0,\lambda_0>0$, and $C^1$-mappings $(\lambda, \theta, x)\colon U_{\alpha_0}\to (1-\lambda_0 ,1+\lambda_0)\times\R^2$ such that 
for all $u\in U_{\alpha_0}$, $\eps(u)\ce  \eps(\lambda(u),\theta(u),x(u);u)$ satisfies
\begin{align}\label{eq:3.1}
  \begin{aligned}
  (\eps(u), \chi_{\om,c})_{L^2}
  &=(\eps(u), i\phi_{\om,c})_{L^2}=(\eps(u), \phi_{\om,c}')_{L^2}=0.
\end{aligned}
\end{align}
Moreover, there exists a constant $C>0$ such that for any $\alpha\in (0,\alpha_0)$ and $u\in U_{\alpha}$
\begin{align} \label{eq:3.2}
  \| \eps(u)\|_{H^1}\le C\alpha,\quad
  |\lambda(u) -1|\le C\alpha.
\end{align}
\end{proposition}

\begin{proof}
Let $F\colon(0,\infty)\times\R^2\times H^1(\R)\to\R$ be the function defined by
\begin{align*}
F(\lambda,\theta,x;u)=
  \l[\begin{array}{@{}c@{}}
  (\eps(\lambda,\theta,x;u),\chi)_{L^2} \\
  (\eps(\lambda,\theta,x;u) ,i\phi)_{L^2} \\
  (\eps(\lambda,\theta,x;u) ,\phi')_{L^2} \end{array}\r].
\end{align*}
We define the open neighborhoods $V_{\alpha}$ of $\phi$ and $\Omega_{\delta}\subset (0,\infty)\times\R^2$ of $(1,0,0)$ by
\begin{align*}
  V_{\alpha}
  &=\{u\in H^1(\R)\colon\, 
  \|u-\phi\|_{H^1}<\alpha\},
\\\Omega_{\delta}
  &=\{(\lambda,\theta,x)\in(0,\infty)\times\R^2\colon\,
  |\lambda -1|+|\theta|+|x|<\delta\}.
\end{align*}
By the orthogonality $\ker L\perp\spn\{\chi\}$ and Lemma~\ref{lem:3.1}, we have
\begin{align*}
  \frac{\del F}{\del(\lambda,\theta,x)}(1,0,0;\phi)
  &=\begin{bmatrix}
  (\Lambda \phi,\chi)_{L^2} & -(i\phi,\chi)_{L^2} & (\phi',\chi)_{L^2} \\
  (\Lambda \phi,i\phi)_{L^2} & -(i\phi,i\phi)_{L^2} & (\phi',i\phi)_{L^2}  \\
  (\Lambda \phi,\phi')_{L^2} & -(i\phi,\phi')_{L^2} & (\phi',\phi')_{L^2}  
  \end{bmatrix}
\\&=\begin{bmatrix}
  (\Lambda \phi,\chi)_{L^2} & 0 & 0 \\
  0 & -\|\phi\|_{L^2}^2 & 0  \\
  0 & 0 & \|\phi'\|_{L^2}^2 
  \end{bmatrix}.
\end{align*}
Since $(\Lambda \phi,\chi)_{L^2}\neq 0$ by Lemma~\ref{lem:3.1}~(3), we deduce that
\begin{align}
\label{eq:3.3}
 \det\frac{\del F}{\del(\lambda,\theta,x)}(1,0,0;\phi)
 \ne 0.
\end{align}
Combined with $F(1,0,0;\phi)=0$, the implicit function theorem implies that there exist constants $\bar{\alpha}>0$, $\bar{\delta}>0$, and $C^1$-mappings $(\lambda,\theta,x)\colon V_{\bar{\alpha}}\to\Omega_{\bar{\delta}}$ such that
\begin{align}
\label{eq:3.4}
  F(\lambda (u), \theta(u), x(u);u) 
  =0\quad
  \text{for all } u\in V_{\bal}
\end{align}
and 
\begin{align}\label{eq:3.5}
  |\lambda(u)-1|
  +|\theta(u)|
  +|x(u)|
  \cleq\|u-\phi\|_{H^1}\quad 
  \text{for all } u\in V_{\bar{\alpha}}. 
\end{align}
By the expression of $\eps(u)$ and \eqref{eq:3.5}, one can compute easily that
\begin{align*}
\|\eps(u)\|_{H^1} \cleq \| u-\phi\|_{H^1} \quad \text{for}\ u\in V_{\bal}.
\end{align*}
In particular, for $\alpha\in(0,\bar{\alpha})$ we have
\begin{align}
\label{eq:3.6}
 \|\eps(u)\|_{H^1}\cleq \alpha,\quad
 |\lambda(u) -1| \cleq \alpha
\quad\text{for}\ u\in V_{\alpha}.
\end{align}
By possibly choosing $\alpha$ smaller, we can extend the functions $\lambda(u)$, $\theta(u)$, and $x(u)$ to the tubular neighborhood $U_\alpha$ (see, e.g., \cite{L09} for more details). This completes the proof.
\end{proof}
\subsection{Control of the modulation parameters}
\label{sec:3.2}

Now we derive the equation for $\eps$ and estimate on the modulation parameters.  

Let $u_0\in U_{\alpha_0}$ and $u(t)$ be the solution of \eqref{DNLSb} with $u(0)=u_0$. We denote the exit times from the tubular neighborhood $U_\alpha$ by
\begin{align*}  
  T_{\alpha}^{\pm}
  &=\inf\{t>0\colon\,
  u(\pm t)\notin U_{\alpha}\}.
\end{align*}
We set $I_{\alpha}=(-T_{\alpha}^{-},T_{\alpha}^{+})$. Since $u(t)\in U_{\alpha_0}$ for $t\in I_{\alpha_0}$, we can define
\begin{align}
\label{eq:3.7}
  \lambda(t)
  \ce\lambda(u(t)),\quad
  \theta(t)
  \ce\theta(u(t)),\quad
  x(t)
  \ce x(u(t)),
\end{align}
where the each function in the right-hand sides is given in Proposition~\ref{prop:3.2}. We see that $\lambda(t)$, $\theta(t)$, and $x(t)$ are $C^1$-functions on $I_{\alpha_0}$. For $t\in I_{\alpha_0}$ we denote
\begin{align}
\label{eq:3.8}
  v(t)
  =v(t,y)
  =\lambda(t)^{1/2}e^{-i\theta(t)}u(t,\lambda(t)y+x(t))
\end{align}
and define the function $\eps(t)$ by
\begin{align}
\label{eq:3.9}
  \eps(t)
  =\ve(\lambda(t),\theta(t),x(t);u(t))=v(t)-\phi_{\om,c}.
\end{align}

We rescale the time as follows. We set
\begin{align*}
  \tilde s(t)
  &=\int_0^t\frac{d\tau}{\lambda(\tau)^2},\quad
  \tilde{I}_{\alpha_0}
  =\tilde s(I_{\alpha_0}).
\end{align*}
Obviously $t\mapsto\tilde s(t)$ is strictly increasing, so the inverse function $\tilde{t}\ce\tilde s^{-1}$ exists. For a function $I_{\alpha_0}\ni t\mapsto f(t)$, we define $\tilde{I}_{\alpha_0}\ni s\mapsto\tilde{f}(s)$ by
\begin{align*}
  \tilde{f}(s)
  =f(\tilde{t}(s)).
\end{align*}
We note that 
\begin{align}
\label{eq:3.10}
 \tilde{f}_s(s) =f_{t}(t)\lambda(t)^2\quad\text{for }s=\tilde{s}(t).
\end{align}
For simplicity of notations, in what follows we omit ``tilde'' over the functions of the variable $s$ although it is the same symbol as the function of the variable $t$.

\begin{lemma}
\label{lem:3.3}
For $s\in I_{\alpha_0}$, $\eps(s)$ satisfies
\begin{align}
\label{eq:3.11}
  i\ve_s
  ={}L\ve
  &+(\theta_s-\om)\phi_{\om,c}
  +\l(\frac{x_s}{\lam}-c\r)i\phi_{\om,c}'
  +\frac{\lambda_s}{\lam}i\Lambda\phi_{\om,c}
\\\notag
  &+(\theta_s-\om)\ve
  +\l(\frac{x_s}{\lam}-c\r)i\ve_y
  +\frac{\lambda_s}{\lam}i\Lambda\ve
  +R(\ve),
\end{align}
where $R(\eps)$ is the sum of second and higher order terms of $\ve$ explicitly written as 
\begin{align*}
  R(\ve)
  ={}&\mathopen{}
  -i|\ve|^2\phi_{\om,c}'-2i\Re(\ve\ol{\phi_{\om,c}})\ve_y
  -4b\{\Re(\ve\ol{\phi_{\om,c}})\}^2\phi_{\om,c} -2b|\phi_{\om,c}|^2|\ve|^2\phi_{\om,c}
\\ \notag
  &-4b|\phi_{\om,c}|^2\Re(\ve\ol{\phi_{\om,c}})\ve
  -i|\ve|^2\ve_y
  -4b|\ve|^2\Re(\ve\ol{\phi_{\om,c}})\phi_{\om,c}-4b\{\Re(\ve\ol{\phi_{\om,c}})\}^2\ve
\\ \notag
  &-2b|\phi_{\om,c}|^2|\ve|^2\ve
  -b|\ve|^4\phi_{\om,c}-4b|\ve|^2\Re(\ve\ol{\phi_{\om,c}})\ve
  -b|\ve|^4\ve,
\end{align*}
and there exists $C>0$ such that 
\begin{equation}\label{eq:3.12}
  \int|R(\ve)|\le C(\|\ve\|_{L^2}^2+\|\ve\|_{L^2}\|\ve_y\|_{L^2})\quad
  \text{for $\ve\in H^1(\R)$ with $\|\ve\|_{H^1}\le 1$.}
\end{equation} 
\end{lemma}
\begin{proof}
By direct calculations we see that $v(t)$ satisfies the equation
\[i\lam^2v_t
  =-v_{yy}-i|v|^2v_y-b|v|^4v
  +\lambda_t\lambda i\Lambda v
  +\theta_t\lam^2v
  +x_t\lambda iv_y. \]
By rescaling the time and \eqref{eq:3.10}, we have 
\begin{align*}
  i v_s
  &=-v_{yy}
  -i|v|^2 v_y
  -b|v|^4 v
  +\frac{ \lambda_s}{\lam}i\Lambda v 
  +\theta_s v
  +\frac{x_s}{{\lam}}iv_y. 
\end{align*}
By substituting $v(s)=\phi+\ve(s)$, we obtain that
\begin{align}\label{eq:3.13}
	i\eps_s
  &=iv_s
  =- v_{yy}
  -i|v|^2 v_y
  -b|v|^4 v
  +\frac{ \lambda_s}{\lam}i\Lambda v 
  +\theta_s v
  +\frac{ x_s}{{\lam}}iv_y
\\&\notag
  =\begin{aligned}[t]
  -(\phi+\ve)_{yy}
  -i|\phi+\ve|^2(\phi+\ve)_y
  -b|\phi+\ve|^4(\phi+\ve)\qquad
\\{}\notag
  +\frac{ \lambda_s}{\lam}i\Lambda(\phi+\ve) 
  +\theta_s(\phi+\ve)
  +\frac{ x_s}{{\lam}}i(\phi+\ve)_y.
  \end{aligned}
 \end{align}
We now set
\begin{align*}
  R_1(\eps)&=
  -i|\phi+\eps|^2(\phi+\eps)_y+i|\phi|^2\phi'+i|\phi|^2\eps_y+2i\Re (\eps\ol{\phi})\phi' 
\\ \notag
  &=-i|\ve|^2\phi'
  -2i\Re(\ve\ol{\phi})\ve_y-i|\ve|^2\ve_y,
\\
  R_2(\eps)&=
  -b|\phi+\eps|^4(\phi+\eps)+b|\phi|^4\phi+b|\phi|^4\eps +4b|\phi|^2\Re(\ve\ol{\phi})\phi
\\ \notag
  &=-b\Bigl(\begin{aligned}[t]
  &4\{\Re(\ve\ol{\phi})\}^2\phi+|\ve|^4\phi+4|\ve|^2\Re(\ve\ol{\phi})\phi+2|\phi|^2|\ve|^2\phi
\\&+4\{\Re(\ve\ol{\phi})\}^2\ve+|\ve|^4\ve
  +4|\phi|^2\Re(\ve\ol{\phi})\ve+4|\ve|^2\Re(\ve\ol{\phi})\ve+2|\phi|^2|\ve|^2\ve \Bigr),
  \end{aligned}
\end{align*}
and $R(\eps)=R_1(\eps)+R_2(\eps)$. By the Sobolev embedding we have
\[\int(|R_1(\eps)|+|R_2(\eps)|)
  \lesssim \|\eps\|_{L^2}^2+\|\ve\|_{L^2}\|\ve_y\|_{L^2}
  \quad
  \text{for $\ve\in H^1(\R)$ with $\|\ve\|_{H^1}\le 1$.}
  \]
From \eqref{eq:3.13}, we obtain that
\begin{align*}
  i\eps_s &=
  \begin{aligned}[t]
  -(\phi+\ve)_{yy}
  &+R_1(\eps)
  -i|\phi|^2\phi'
  -i|\phi|^2\eps_y
  -2i\Re (\eps\ol{\phi})\phi' 
\\&+R_2(\eps)
  -b|\phi|^4\phi
  -3b|\phi|^4\eps
  -2b|\phi|^2\phi^2\ol{\eps}
\\ \notag 
  &+\frac{ \lambda_s}{\lam}i\Lambda(\phi+\ve) 
  +\theta_s(\phi+\ve)
  +\frac{ x_s}{\lam}i(\phi+\ve)_y
\end{aligned}
\\&=\begin{aligned}[t]
  &\mathopen{}-\ve_{yy}-i|\phi|^2\eps_y-2i\Re (\eps\ol{\phi})\phi'-3b|\phi|^4\eps -2b|\phi|^2\phi^2\ol{\eps}
\\&-\phi'' -i|\phi|^2\phi' -b|\phi|^4\phi
  +\frac{ \lambda_s}{\lam}i\Lambda(\phi+\ve) 
  +\theta_s(\phi+\ve)
  +\frac{x_s}{\lam}i(\phi+\ve)_y+R(\eps).
\end{aligned}
\end{align*}
By using the relations
\begin{align*}
  &-\ve_{yy}-i|\phi|^2\eps_y-2i\Re (\eps\ol{\phi})\phi'-3b|\phi|^4\eps -2b|\phi|^2\phi^2\ol{\eps}
  =L\eps -\omega\eps-ci\eps_y,
\\&-\phi''-i|\phi|^2\phi' -b|\phi|^4\phi
  =-\omega\phi-ci\phi',
\end{align*}
we obtain \eqref{eq:3.11}. 
\end{proof}
We note that from Proposition~\ref{prop:3.2}, 
\begin{align}
  \label{eq:3.14}
  &(\eps(s),\chi_{\om,c})_{L^2}
  =(\eps(s), i\phi_{\om,c})_{L^2}
  =(\eps(s),\phi_{\om,c}')_{L^2}
  =0,
\\ \label{eq:3.15}
  &\|\eps(s)\|_{H^1}\leq C\alpha,\quad
  |\lambda(s)-1|\leq C\alpha
\end{align}
hold for $\alpha\in(0,\alpha_0)$ and $s\in I_{\alpha}$, where $C$ is independent of $\alpha$ and $s$. 
\begin{lemma}
\label{lem:3.4}
Let $b\ge0$ and $c=2\kappa_0\sqrt{\omega}$. For $s\in I_{\alpha_0}$, the following equalities hold.
\begin{align*}
  \frac{\lambda_s}{\lam}(\Lambda\phi_{\om,c},\chi_{\om,c})_{L^2}
  ={}&\mathopen{}-(\ve,L_{\om,c}i\chi_{\om,c})_{L^2}
  -(\theta_s-\om)(\ve,i\chi_{\om,c})_{L^2}
\\& \notag 
  +\l(\frac{x_s}{\lam}-c\r)(\ve,\chi_{\om,c}')_{L^2}
  +\frac{\lambda_s}{\lam}(\ve,\Lambda\chi_{\om,c})_{L^2}
  -(R(\ve),i\chi_{\om,c})_{L^2},
\\[7pt] 
  (\theta_s-\om)\|\phi_{\om,c}\|_{L^2}^2
  ={}&\mathopen{}-(\ve,L_{\om,c}\phi_{\om,c})_{L^2}
  -(\theta_s-\om)(\ve,\phi_{\om,c})_{L^2}
\\& \notag
  -\l(\frac{x_s}{\lam}-c\r)(\ve,i\phi_{\om,c}')_{L^2}
  -\frac{\lambda_s}{\lam}(\ve,\Lambda i\phi_{\om,c})_{L^2}
  -(R(\ve),\phi_{\om,c})_{L^2},
\\[7pt] 
{\l(\frac{x_s}{\lam}-c\r)}\|\phi_{\om,c}'\|_{L^2}^2
  ={}&\mathopen{}-(\ve,L_{\om,c}i\phi_{\om,c}')_{L^2}
  -(\theta_s-\om)(\ve,i\phi_{\om,c}')_{L^2}
\\& \notag
  +\l(\frac{x_s}{\lam}-c\r)(\ve,\phi_{\om,c}'')_{L^2}
  +\frac{\lambda_s}{\lam}(\ve,\Lambda\phi_{\om,c}')_{L^2}
  -(R(\ve),i\phi_{\om,c}')_{L^2}.
\end{align*}
Moreover, there exist $C>0$ and $\alpha_1\in(0,\alpha_0)$ such that for $s\in I_{\alpha_1}$, the following estimate holds.
\begin{align}
 \label{eq:3.16} 
  \l|\frac{\lam_s}{\lam}\r|
  +\l|\theta_s-\om\r|
  +\l|\frac{x_s}{\lam}-c\r|\le C\|\ve(s)\|_{L^2}.
\end{align}
\end{lemma}
\begin{proof}
By differentiating the orthogonal relation $(\eps(s),\chi)_{L^2}=0$ with respect to $s$, we have the first relation in the statement as follows:
\begin{align*}
  0={}&(\ve_s,\chi)_{L^2}
\\={}&\mathopen{}-(iL\ve,\chi)_{L^2}
  -(\theta_s-\om)(i\phi,\chi)_{L^2}
  +\l(\frac{x_s}{\lam}-c\r)(\phi',\chi)_{L^2}
  +\frac{\lambda_s}{\lam}(\Lambda\phi,\chi)_{L^2}
\\&-(\theta_s-\om)(i\ve,\chi)_{L^2}
  +\l(\frac{x_s}{\lam}-c\r)(\ve_y,\chi)_{L^2}
  +\frac{\lambda_s}{\lam}(\Lambda\ve,\chi)_{L^2}
  -(iR(\ve),\chi)_{L^2}
\\={}&\mathopen{}(\ve,Li\chi)_{L^2}
  +\frac{\lambda_s}{\lam}(\Lambda\phi,\chi)_{L^2}
\\&+(\theta_s-\om)(\ve,i\chi)_{L^2}
  -\l(\frac{x_s}{\lam}-c\r)(\ve,\chi')_{L^2}
  -\frac{\lambda_s}{\lam}(\ve,\Lambda\chi)_{L^2}
  +(R(\ve),i\chi)_{L^2},
\end{align*}
where we used $(i\phi,\chi)_{L^2}=(\phi',\chi)_{L^2}=0$ in the last equality. 

From Lemma~\ref{lem:3.1} we recall that the following equalities hold.
\begin{align*}
(\Lambda\phi ,i\phi)_{L^2}=(\Lambda\phi,\phi')_{L^2}=(i\phi',\phi)_{L^2}=0.
\end{align*}
By differentiating the relation $(\ve(s),i\phi)_{L^2}=0$ with respect to $s$, we obtain the second relation as
\begin{align*}
  0={}&(\ve_s,i\phi)_{L^2}
\\={}&\mathopen-(iL\ve,i\phi)_{L^2}
  -(\theta_s-\om)(i\phi,i\phi)_{L^2}
  +\l(\frac{x_s}{\lam}-c\r)(\phi',i\phi)_{L^2}
  +\frac{\lambda_s}{\lam}(\Lambda\phi,i\phi)_{L^2}
\\&-(\theta_s-\om)(i\ve,i\phi)_{L^2}
  +\l(\frac{x_s}{\lam}-c\r)(\ve_y,i\phi)_{L^2}
  +\frac{\lambda_s}{\lam}(\Lambda\ve,i\phi)_{L^2}
  -(iR(\ve),i\phi)_{L^2}
\\={}&\mathopen{}-(\ve,L\phi)_{L^2}
  -(\theta_s-\om)\|\phi\|_{L^2}^2
\\&-(\theta_s-\om)(\ve,\phi)_{L^2}
  -\l(\frac{x_s}{\lam}-c\r)(\ve,i\phi')_{L^2}
  -\frac{\lambda_s}{\lam}(\ve,i\Lambda\phi)_{L^2}
  -(R(\ve),\phi)_{L^2}.
\end{align*}
Similarly, by differentiating the relation $(\ve(s),\phi')_{L^2}=0$ with respect to $s$, we obtain the third relation as
\begin{align*}
  0={}&(\ve_s,\phi')_{L^2}
\\={}&-(iL\ve,\phi')_{L^2}
  -(\theta_s-\om)(i\phi,\phi')_{L^2}
  +\l(\frac{x_s}{\lam}-c\r)(\phi',\phi')_{L^2}
  +\frac{\lambda_s}{\lam}(\Lambda\phi,\phi')_{L^2}
\\&-(\theta_s-\om)(i\ve,\phi')_{L^2}
  +\l(\frac{x_s}{\lam}-c\r)(\ve_y,\phi')_{L^2}
  +\frac{\lambda_s}{\lam}(\Lambda\ve,\phi')_{L^2}
  -(iR(\ve),\phi')_{L^2}
\\={}&(\ve,Li\phi')_{L^2}
  +\l(\frac{x_s}{\lam}-c\r)\|\phi'\|_{L^2}^2
\\&+(\theta_s-\om)(\ve,i\phi')_{L^2}
  -\l(\frac{x_s}{\lam}-c\r)(\ve,\phi'')_{L^2}
  -\frac{\lambda_s}{\lam}(\ve,\Lambda\phi')_{L^2}
  +(R(\ve),i\phi')_{L^2}.
\end{align*}
From three relations above and \eqref{eq:3.12}, we obtain
\begin{align*}
  \l|\frac{\lam_s}{\lam}\r|+\l|\theta_s-\om\r|+\l|\frac{x_s}{\lam}-c\r|
  \lesssim \|\ve\|_{L^2}
  +\l( \l|\frac{\lam_s}{\lam}\r|
  +\l|\theta_s-\om\r|+\l|\frac{x_s}{\lam}-c\r|\r)\| \eps \|_{L^2}.
\end{align*}
By \eqref{eq:3.15} and taking $\alpha$ small enough, we obtain the estimate \eqref{eq:3.16}. 
\end{proof}

\subsection{Error estimates}\label{sec:3.3}
In this subsection, we derive the uniform estimate of $\eps(s)$ for $s\in I_{\alpha_0}$. Assume that $\eps_0\in H^1(\R)$ satisfies 
\begin{align}
\label{eq:3.17}
  (\eps_0,\chi_{\om,c})_{L^2}
  =(\eps_0, i\phi_{\om,c})_{L^2}
  =(\eps_0,\phi_{\om,c}')=0.
\end{align}
We set $u_0=\phi_{\om,c}+\eps_0$. From \eqref{eq:3.4} and \eqref{eq:3.17}, we have
\begin{align*}
  \lambda(0)
  =\lambda(u_0)=1,\quad
  \theta(0)=\theta(u_0)=0,\quad
  x(0)=x(u_0)=0,
\end{align*}
which implies that
\begin{align*}
  \eps(0)
  =\eps(\lambda(0),\theta(0),x(0); u(0))
  =\eps(1,0,0; u_0)
  =u_0-\phi_{\om,c}
  =\eps_0.
\end{align*}
We define
\begin{equation*}
\begin{aligned}
  E_e(\ve)
  &=E(\phi_{\om,c}+\ve)
  -E(\phi),
\\M_e(\ve)
  &=M(\phi_{\om,c}+\ve)
  -M(\phi_{\om,c})
  =2(\phi_{\om,c},\ve)_{L^2}
  +M(\ve),
\\P_e(\ve)
  &=P(\phi_{\om,c}+\ve)
  -P(\phi_{\om,c})
  =2(i\phi_{\om,c}',\ve)
  +P(\ve),
\\S_e(\ve)
  &=S_{\om,c}(\phi+\ve)-S_{\om,c}(\phi)
  =E_e(\ve)+\frac{\omega}{2}M_e(\ve)+\frac{c}{2}P_e(\ve).
\end{aligned}
\end{equation*}
\begin{lemma}\label{lem:3.5}
For $\ve\in H^1(\R)$, we have
\begin{align*}
  E_e(\ve)
  &=-\omega(\phi_{\om,c},\ve)_{L^2}
  -c(i\phi_{\om,c}',\ve)_{L^2}
  +O(\|\ve\|_{H^1}^2),
\\M_e(\ve)
  &=2(\phi_{\om,c},\ve)_{L^2}
  +O(\|\ve\|_{H^1}^2),
\\P_e(\ve)
  &=2(i\phi_{\om,c}',\ve)_{L^2}
  +O(\|\ve\|_{H^1}^2),
\\S_e(\ve)
  &=\frac12\langle L_{\om,c}\ve,\ve\rangle
  +O(\|\ve\|_{H^1}^3)
  =O(\|\ve\|_{H^1}^2).
\end{align*}
\end{lemma}
\begin{proof}
Since $S'(\phi)=0$, this is equivalent to
\[E'(\phi)
  =-\omega\phi-ci\phi'. \]
By the Taylor expansion we have
\begin{align*}
  E_e(\ve)
  &=E(\phi+\ve)-E(\phi)
  =\langle E'(\phi),\ve\rangle
  +O(\|\ve\|_{H^1}^2)
\\&=-\omega(\phi,\ve)_{L^2}
  -c(i\phi',\ve)_{L^2}
  +O(\|\ve\|_{H^1}^2),
\\
S_e(\ve)
  &=S(\phi+\ve)-S(\phi)
  =\frac12\langle L\ve,\ve\rangle+O(\|\ve\|_{H^1}^3).
\end{align*}
The estimates for $M_e$ and $P_e$ are trivial from the definition.
\end{proof}
\begin{lemma}\label{lem:3.6}
Let $b\ge 0$ and $c=2\kappa_0\sqrt{\omega}$. For $s\in I_{\alpha_0}$, we have
\begin{align*}
  &M_e(\ve(s))
  =M_e(\ve_0),\quad
  P_e(\ve(s))
  =\lambda(s)P_e(\ve_0),\quad
  E_e(\ve(s))
  =\lambda(s)^2E_e(\ve_0).
\end{align*}
\end{lemma}
\begin{proof}
A direct computation shows that
\begin{align*}
  M(\phi+\ve(s))
  &=M(v(s))
  =M(u(s))
  =M(u_0)
  =M(\phi+\ve_0).
\end{align*}
By expanding both sides we deduce that 
\[
  2(\phi,\ve(s))_{L^2}
  +M(\ve(s))
  =2(\phi,\ve_0)_{L^2}
  +M(\ve_0), \]
which is the desired equality. 

Since $E(\phi)=P(\phi)=0$ from the assumption, we have
\begin{align*}
  E_e(\ve(s))=E(\phi_{\om,c}+\ve(s))=E(v(s)),\quad
  P_e(\ve(s))=P(\phi_{\om,c}+\ve(s))=P(v(s)).
\end{align*}
Therefore, we deduce that
\begin{align*}
  P_e(\ve(s))
  &=P(v(s))
  =\lambda(s)P(u(t(s)))
  =\lambda(s)P(u_0)
  =\lambda(s)P_e(\ve_0),
\\E_e(\ve(s))
  &=E(v(s))
  =\lambda(s)^2E(u(t(s)))
  =\lambda(s)^2E(u_0)
  =\lambda(s)^2E_e(\ve_0).
\end{align*}
This completes the proof.
\end{proof}
\begin{lemma}
\label{lem:3.7}
Let $b>0$ and $c=2\kappa_0\sqrt{\omega}$. Then there exist $C>0$ and $\alpha_2\in(0,\alpha_0)$ such that for any $\alpha\in(0,\alpha_2)$ and $s\in I_{\alpha}$, we have
\begin{align}
\label{eq:3.18}
  \|\ve(s)\|_{H^1}^2
  \le{}C\bigl(&\alpha|2\omega(\phi_{\om,c},\ve_0)_{L^2}+c(i\phi_{\om,c}',\ve_0)_{L^2}|
\\&+\alpha^2|\omega(\phi_{\om,c},\ve_0)_{L^2}+c(i\phi_{\om,c}',\ve_0)_{L^2}|
  +\|\ve_0\|_{H^1}^2\bigr).
 \notag 
\end{align}
\end{lemma}
\begin{proof}
Since $\omega>c^2/4$ from the assumption, we note that the coercivity property \eqref{eq:1.9} holds. 
It follows from Lemma~\ref{lem:3.5} and \eqref{eq:3.15} that by taking $\alpha$ small enough,
\[S_e(\ve(s))
  =\frac12\langle L\ve(s),\ve(s)\rangle
  +O(\|\ve(s)\|_{H^1}^3)
  \gtrsim \|\ve(s)\|_{H^1}^2. \]
On the other hand, we deduce from Lemmas~\ref{lem:3.5} and \ref{lem:3.6} that
\begin{align*}
  S_e(\ve(s))
  ={}&\lambda(s)^2E_e(\ve_0)
  +\frac{\omega}{2}M_e(\ve_0)
  +\lambda(s)\frac{c}{2}P_e(\ve_0)
\\={}&S_e(\ve_0)
  +(\lambda(s)^2-1)E_e(\ve_0)
  +(\lambda(s)-1)\frac{c}{2}P_e(\ve_0)
\\={}&(\lambda(s)-1)\Bigl(2E_e(\ve_0)
  +\frac{c}{2}P_e(\ve_0)\Bigr)
  +(\lambda(s)-1)^2E_e(\ve_0)+O(\| \eps_0\|_{H^1}^2)
\\={}&(\lambda(s)-1)\bigl(-2\omega(\phi,\ve_0)_{L^2}-c(i\phi',\ve_0)_{L^2}\bigr)
\\&-(\lambda(s)-1)^2\bigl(\omega(\phi,\ve_0)_{L^2}+c(i\phi',\ve_0)_{L^2}\bigr)
  +O(\|\ve_0\|_{H^1}^2).
\end{align*}
Therefore, combined with \eqref{eq:3.15}, we obtain \eqref{eq:3.18}. 
\end{proof}

\section{Virial identities}
\label{sec:4}

In this section we organize virial identities of \eqref{DNLSb}.
Let $u$ be the $H^1$-solution of \eqref{DNLSb} with $u(0)=u_0\in H^1(\R)$, which is defined on a maximal interval $(-T_{\rm min},T_{\rm max})$. 
\begin{proposition}[Virial identity] \label{prop:4.1}
For $u_0\in H^1(\R)$ such that $\int x^2|u_0|^2<\infty$, we have the following relations:
\begin{align}
\label{eq:4.1}
\frac{d}{dt}\int x^2|u|^2&=4\Im\int xu_x\ol{u} +\int x|u|^4,
\\
\label{eq:4.2}
\frac{d}{dt}\Im\int xu_x\ol{u} &=4E(u_0)
\end{align}
for $t\in (-T_{\rm min},T_{\rm max})$.
\end{proposition}
\begin{proof}
See \cite[Lemma 2.2]{Wu13} and \cite[Proposition 6.5.1]{C03}.
\end{proof}
The first relation \eqref{eq:4.1} is different from the one of (NLS) due to the appearance of the second term in the right-hand side. On the other hand, the second relation \eqref{eq:4.2} is the same as (NLS). We take advantage of the latter relation for the proof of instability.

We now assume that $u(0)=u_0\in U_{\alpha_0}$. We recall that $v(t)$ and $\eps(t)$ are defined in \eqref{eq:3.8} and \eqref{eq:3.9}, respectively. We rescale the time variable $t$ to $s$ as in Section~\ref{sec:3}. Following \cite{MR05}, we rewrite the virial relation in terms of $\ve(s)$. We denote
\[J[v]=\Im\int y v_y\ol{v}\,dy
  =-\Re\int iy v_y\ol{v}\,dy. \]
Then $J[\eps]$ is represented as follows.
\begin{lemma}
Let $b\ge0$ and $c=2\kappa_0\sqrt{\omega}$.
Assume that $\int x^2|u_0|^2<\infty$. For $s\in I_{\alpha_0}$, we have
\begin{align} \label{eq:4.3}
  J[\eps(s)]
  =2(\ve(s),i\Lambda\phi_{\omega,c})_{L^2}
  +J[u(s)]+x(s)P(u_0).
\end{align}
\end{lemma}

\begin{proof}
From the phase and scaling invariance of $J$, we have
\begin{align}\label{eq:4.4}
  J[v(s)]
  &=J[u(s,\cdot+x(s)]
 =J[u(s)]+x(s)P(u_0).
\end{align}
On the other hand, $J[v(s)]$ is rewritten as
\begin{align} 
\label{eq:4.5}
  J[v(s)]
  =J[\ve(s)+\phi]
  =J[\ve(s)]-2(\ve,i\Lambda\phi)_{L^2}+J[\phi].
\end{align}
By Lemma~\ref{lem:3.1}, $J[\phi]$ is rewritten as
\begin{align} 
\label{eq:4.6}
  J[\phi]
  =(i\phi,y\phi')_{L^2}
  =(i\phi,\tfrac{1}{2}\phi+y\phi')_{L^2}
  =(i\phi,\Lambda\phi)_{L^2}
  =0.
\end{align}
By combining \eqref{eq:4.4}, \eqref{eq:4.5}, and \eqref{eq:4.6}, we obtain \eqref{eq:4.3}.
\end{proof}
The first term in the right-hand side of \eqref{eq:4.3} 
\begin{align}
\label{eq:4.7}
  (\ve(s),i\Lambda\phi_{\omega ,c})_{L^2}
=\Im\int \eps(s)\Lambda\phi_{\omega,c}
\end{align}
plays an essential role in our proof of instability. We note that \eqref{eq:4.7} is well-defined without the assumption $\int x^2|u_0|^2<\infty$. 
From the equation \eqref{eq:3.11}, we have
\begin{align*}
  \notag
  \frac{d}{ds}(\ve(s),i\Lambda\phi)_{L^2}
  &=-(i\ve_s(s),\Lambda\phi)_{L^2}
\\ \notag
&=-\biggl(
  \begin{aligned}[t]
  &L\ve
  +(\theta_s-\om)\phi +i\l(\frac{x_s}{\lam}-c\r)\phi'+i\frac{\lambda_s}{\lam}\Lambda\phi 
\\&+(\theta_s-\om)\ve+i\l(\frac{x_s}{\lam}-c\r)\ve_y+i\frac{\lambda_s}{\lam}\Lambda\ve
  +R(\ve),\Lambda\phi\biggr)_{L^2}
  \end{aligned}
\end{align*}
for $s\in I_{\alpha_0}$. We note that $(\phi,\Lambda\phi)_{L^2}=(i\Lambda\phi,\Lambda\phi)_{L^2}=0$ and $(i\phi', \Lambda\phi)_{L^2}=0$ by Lemma~\ref{lem:3.1}~(2). Therefore, by \eqref{eq:3.12} and \eqref{eq:3.16}, we deduce that
\begin{align}
\label{eq:4.8}
  \frac{d}{ds}(\ve(s),i\Lambda\phi)_{L^2}
  &=-(\ve(s),L\Lambda\phi)_{L^2}
  +O(\|\ve(s)\|_{H^1}^2)
\end{align}
for $s\in I_{\alpha_1}$, where $\alpha_1>0$ appeared in Lemma \ref{lem:3.4}. Therefore, by using the relation
$L\Lambda\phi=-2\omega\phi -ci\phi'$,
we obtain the following claim.
\begin{lemma}
\label{lem:4.3}
Let $b\ge0$ and $c=2\kappa_0\sqrt{\omega}$. There exists $C>0$ such that for $s\in I_{\alpha_1}$,
\begin{align}
\label{eq:4.9}
  \left|\frac{d}{ds}(\ve(s),i\Lambda\phi_{\om,c})_{L^2}
  -(\ve(s),2\omega\phi_{\om,c}+ci\phi_{\om,c}')_{L^2}\right|
  \le C\|\ve(s)\|_{H^1}^2.
\end{align}
\end{lemma}

\section{Proof of instability} \label{sec:5}

We are now in a position to complete the proof of Theorem~\ref{thm:1.3}.
We first note that by Lemma~\ref{lem:3.5}, the second term in the left-hand side of \eqref{eq:4.9} is rewritten as
\begin{align*}
  (\eps(s),2\omega\phi+ci\phi')_{L^2}
  &=\omega M_e(\eps(s))
  +\frac{c}{2}P_e(\eps(s))
  +O(\|\eps(s)\|_{H^1}^2).
\end{align*}
By Lemma~\ref{lem:3.6} we have
\begin{align*}
  \omega M_e(\eps(s))+\frac{c}{2}P_e(\eps(s))
  &=\omega M_e(\eps_0)+\frac{c}{2}\lambda(s)P_e(\eps_0)
\\&=2\omega(\eps_0,\phi)_{L^2}
  +c\lambda(s)(\eps_0,i\phi')_{L^2}
  +O(\|\eps_0\|_{H^1}^2).
\end{align*}
Therefore, we obtain the following expression:
\begin{align}
  \label{eq:5.1}
  (\ve(s),2\omega\phi+ci\phi')_{L^2}
  ={}&2\omega(\eps_0,\phi)_{L^2}
  +c\lambda(s)(\eps_0,i\phi')_{L^2}
\\\notag&\quad
  +O(\|\eps_0\|_{H^1}^2)+O(\|\eps(s)\|_{H^1}^2).
\end{align}

\begin{proof}[Proof of Theorem~\ref{thm:1.3}]
We proceed by contradiction. Suppose that for each $\alpha,\beta\in(0,1)$ there exists $u_0=u_{0,\alpha,\beta}\in H^1(\R)$ such that $\ve_0\ce u_0-\phi_{\om,c}$ satisfies 
\begin{align*}
  &0<\|\ve_0\|_{H^1}^2\le \beta \l|(\ve_0,\phi)_{L^2}\r|,&
  &\ve_0\perp\{\chi,i\phi,\phi',i\phi'\},&
\end{align*}
and that the solution $u(t)$ of \eqref{DNLSb} with $u(0)=u_0$ satisfies $u(t)\in U_\alpha$ for all $t\in\R$. It follows that $I_{\alpha}=\R$. 

Let $\alpha,\beta>0$ to be chosen later. For now we take $\alpha$ small enough so that
\begin{align*}
  0<\alpha
  <\min\{\alpha_1, \alpha_2\}<\alpha_0<1.
\end{align*}
In what follows, we only consider the case $(\ve_0,\phi)_{L^2}>0$ because one can treat the case $(\ve_0,\phi)_{L^2}<0$ in the same way. From Lemma~\ref{lem:3.7}, we have
\begin{align*}
  \sup_{s\in\R}\| \eps(s)\|_{H^1}^2
  \cleq\alpha(\eps_0,\phi)_{L^2}
  +\|\eps_0\|_{H^1}^2.
\end{align*}
We note that $\sup_{s\in\R}|\lambda(s)-1|\cleq\alpha$ by \eqref{eq:3.15}. Therefore, by Lemma~\ref{lem:4.3} and \eqref{eq:5.1}, we have
\begin{align*}
  \frac{d}{ds}(\ve(s),i\Lambda\phi)_{L^2}
  &\cgeq(\eps_0,\phi)_{L^2}
  -\alpha(\eps_0,\phi)_{L^2}
  +O(\|\eps_0\|_{H^1}^2)
\\&\cgeq(1-\alpha-C\beta)(\eps_0,\phi)_{L^2},
\end{align*}
where the constant $C$ is independent of $\eps_0$, $\alpha$, $\beta$, and $s$. Therefore, by taking $\alpha,\beta>0$ small enough, we obtain 
\begin{align*}
  \frac{d}{ds}(\ve(s),i\Lambda\phi)_{L^2}
  &\cgeq(\eps_0,\phi)_{L^2}>0
\end{align*}
for all $s\in\R$. This uniform estimate yields that
\[(\ve(s),i\Lambda\phi)_{L^2}
  \to\infty\quad
  \text{as}\ s\to\infty. \]
On the other hand, from \eqref{eq:3.15} we have the bound
\begin{align*}
  \sup_{s\in\R}\l|(\ve(s),i\Lambda\phi)_{L^2}\r| 
  \cleq \|\Lambda\phi\|_{L^2}
  <\infty,
\end{align*}
which is a contradiction. This completes the proof.
\end{proof}

\appendix
\section{Relation to instability theory on \eqref{eq:gKdV}} 
\label{sec:A}%
By following the argument of \cite{FHR19}, we review the instability theory of the soliton $Q(\cdot -t)$ for the $L^2$-critical generalized KdV equation
\begin{align*}
\tag{gKdV}
u_t+( u_{xx}+u^5)_x =0,\quad (t,x)\in \R\times\R,
\end{align*}
and see a relation to our proof of Theorem \ref{thm:1.3}. 

We define a tubular neighborhood around $Q$ by
\begin{align*}
  U_{\alpha}=\{u\in H^1(\R)\colon\,\inf_{y\in\R}\| u-Q(\cdot -y)\|_{H^1}<\alpha\}.
\end{align*}
The linearized operator $L$ around $Q$ is given by
\begin{align*}
Lv =-v_{xx}+v -5Q^4v\quad\text{for}~v\in H^1(\R).
\end{align*}
We note that $L$ satisfies the following properties:
\begin{align*}
LQ^3=-8Q^3,\quad\ker L =\spn\{Q' \}.
\end{align*}
We consider the initial data $u_0=Q+\eps_0$ such that $\eps_0\in H^1(\R)$ satisfies
\begin{align}
\label{eq:A.1}
  (\eps_0,Q^3)_{L^2}
  =(\eps_0,Q')_{L^2}=0.
\end{align}
Let $u(t)$ be the solution of \eqref{eq:gKdV} with $u(0)=u_0$. 
In the same way as in Section \ref{sec:3}, one can prove that there exist $\alpha_0>0$ and $C^1$-functions $\lambda(t)>0$ and $x(t)\in\R$ such that if $u(t)\in U_{\alpha_0}$ for all $t\ge 0$, then $\eps(t)=\eps(t,y)$ defined by
\begin{align*}
\eps(t,y)=\lambda(t)^{1/2}u(t,\lambda(t)y+x(t))-Q(y)
\end{align*}
satisfies 
\begin{align}
\label{eq:A.2}
  (\eps(t) ,Q^3)_{L^2}=(\eps(t) ,Q')_{L^2}=0\quad
  \text{for all}\ t\ge 0. 
\end{align}
We rescale the time $t\mapsto s$ by $\frac{ds}{dt}=\frac{1}{\lambda^3(t)}$. A direct calculation shows that $\eps(s)$ satisfies
\begin{align}
\label{eq:A.3}
	\eps_s 
   =(L\eps)_y+\frac{\lambda_s}{\lambda}\Lambda Q+\l(\frac{ x_s}{\lambda}-1\r) Q_y+\frac{\lambda_s}{\lambda}\Lambda\eps+\l(\frac{ x_s}{\lambda}-1\r)\eps_y-r(\eps)_y,
\end{align}
where $r(\eps)$ is the sum of second and higher order terms of $\eps$. By \eqref{eq:A.2} and \eqref{eq:A.3} one can prove that
\begin{align}
\label{eq:A.4}
\l| \frac{\lambda_s}{\lambda} \r|+\l| \frac{x_s}{\lambda}-1\r|
\cleq \|\eps(s)\|_{L^2}
\quad\text{for all}~s\ge 0. 
\end{align}

We now introduce the following functional
\begin{align}
\label{eq:A.5}
J (s) =\int \eps(s)\int_{-\infty}^{y}\Lambda Q,
\end{align}
which corresponds to \eqref{eq:4.7} as a Lyapunov functional.
As pointed out in \cite{FHR19}, if we consider the exponentially decaying data as
\begin{align}
\label{eq:A.6}
|\eps_0 (x)| \cleq ce^{-\delta|x|}\quad\text{for some}~\delta>0,
\end{align}
it is rather easy to show the $L^2$-exponential decay on the right of the soliton. In particular, \eqref{eq:A.5} is well-defined for all $s\ge 0$.
From \eqref{eq:A.3} and \eqref{eq:A.4}, one can obtain easily that
\begin{equation}
\label{eq:A.7}
    \frac{d}{ds}J(s)= -\int\eps(s)L\Lambda Q-\frac{\lambda_s}{2\lambda}\l(J(s)-\frac{1}{4}\l(\int Q\r)^2\r)
   +O(\| \eps(s)\|_{L^2}^2).
\end{equation}
Here we define a rescaled functional of $J$ by
\begin{align*}
K(s)=\lambda(s)^{1/2}\l(J(s)- \frac{1}{4}\l(\int Q\r)^2 \r).
\end{align*}
It follows from \eqref{eq:A.7} that
\begin{align*}
\frac{d}{ds}K(s)&=-\lambda(s)^{1/2}\int \eps(s)L\Lambda Q +O(\| \eps(s)\|_{L^2}^2),
\end{align*}
which corresponds to \eqref{eq:4.8}. By using the relation $L\Lambda Q=-2Q$, we have
\begin{align*}
\frac{d}{ds}K(s)&=2\lambda(s)^{1/2}\int \eps(s)Q +O(\| \eps(s)\|_{L^2}^2),
\end{align*}
which corresponds to \eqref{eq:4.9}. Therefore, if we assume \eqref{eq:A.1}, \eqref{eq:A.6} and
\begin{align}
\label{eq:A.8}
0<\|\eps_0\|_{H^1}^2\le b_0\int \eps_0Q
\end{align}
for suitably small $b_0>0$, we can complete the proof of instability of the soliton.

We conclude that the functionals \eqref{eq:4.7} and \eqref{eq:A.5} play an essential role in the proof of instability of the degenerate solitons in \eqref{DNLSb} and \eqref{eq:gKdV}, respectively, and that the unstable directions are determined by $L\Lambda\phi$ for \eqref{DNLSb} and $L\Lambda Q$ for \eqref{eq:gKdV}, respectively.

\section*{Acknowledgments}
N.F. was supported by JSPS KAKENHI Grant Number JP20K14349 and  M.H. by JSPS KAKENHI Grant Number JP19J01504.

\end{document}